\documentclass[11pt,reqno]{article}
\usepackage{dsfont, amssymb,amsmath,amscd,latexsym, amsthm, extarrows, amsxtra,amsfonts}
\usepackage[all]{xy}
\usepackage[active]{srcltx}
\usepackage[pdfstartview=FitH, bookmarksnumbered=true,bookmarksopen=true, colorlinks=true, pdfborder=001, citecolor=blue, linkcolor=blue,urlcolor=blue]{hyperref}
\usepackage[round,authoryear]{natbib}
\usepackage{graphicx}  
\usepackage{float}  
\usepackage{subfigure}

\usepackage{graphicx}
\usepackage[title]{appendix}
\usepackage{bm}
\usepackage{tikz}
\usepackage[top=2.8cm, bottom=2.8cm, left=2.8cm, right=2.8cm]{
	geometry}
\newtheorem{theorem}{Theorem}[section]

\newtheorem{corollary}[theorem]{Corollary}
\newtheorem{definition}[theorem]{Definition}

\newtheorem{lemma}[theorem]{Lemma}

\newtheorem{proposition}[theorem]{Proposition}
\newtheorem{remark}[theorem]{Remark}
\numberwithin{equation}{section}

\begin{document}
\makeatletter
\def\@setauthors{%
\begingroup
\def\thanks{\protect\thanks@warning}%
\trivlist \centering\footnotesize \@topsep30\p@\relax
\advance\@topsep by -\baselineskip
\item\relax
\author@andify\authors
\def\\{\protect\linebreak}%
{\authors}%
\ifx\@empty\contribs \else ,\penalty-3 \space \@setcontribs
\@closetoccontribs \fi
\endtrivlist
\endgroup } \makeatother
 \baselineskip 18pt

\title{Robust Portfolio Selection under State-dependent Confidence Set}
\author{
Guohui Guan\thanks{School of Statistics, Renmin University of China, Beijing 100872, China; Email: $<$guangh@ruc.edu.cn$>$.}
\and
Yuting Jia\thanks{Department of Mathematical Sciences, Tsinghua
University, Beijing 100084, China; Email: $<$jyt22@mails.tsinghua.edu.cn$>$}
\and
Zongxia Liang\thanks{Department of Mathematical Sciences, Tsinghua
University, Beijing 100084, China; Email: $<$liangzongxia@mail.tsinghua.edu.cn$>$.}
}

\maketitle

\begin{abstract}
This paper studies the robust portfolio selection problem under a state-dependent confidence set. The investor invests in a financial market with a risk-free asset and a risky asset. The ambiguity-averse investor faces uncertainty over the drift of the risky asset and updates posterior beliefs by Bayesian learning. The investor holds the belief that the unknown drift falls within a confidence set at a certain confidence level. The confidence set varies with both the observed state and time. By maximizing the expected CARA utility of terminal wealth under the worst-case scenario of the unknown drift, we derive and solve the associated Hamilton–Jacobi–Bellman–Isaacs (HJBI) equation. The robust optimal investment strategy is obtained in a semi-analytical form  based on a partial differential equation (PDE). We validate the existence and uniqueness of the PDE and demonstrate the optimality of the solution in the verification theorem. The robust optimal investment strategy consists of two components: myopic demand in the worst-case scenario and hedging demand. The robust optimal investment strategy is categorized into three regions: buying, selling, and small trading.  Ambiguity aversion results in a more conservative robust optimal investment strategy. Additionally, with learning, the investor's uncertainty about the drift decreases over time, leading to increased risk exposure to the risky asset.
 \vskip 15 pt \noindent
\textbf{Keywords:}  Robust portfolio selection; Ambiguity; Bayesian learning; State-dependent confidence set.
\end{abstract}
\vskip15pt

\section{Introduction}\label{Introduction}
Portfolio selection is a fundamental topic in modern financial theory. \citet{markowitz1952} lays the groundwork for modern portfolio theory with the one-period mean-variance model. Later, \citet{merton1969life,merton1971optim} and \citet{samuelson1969life} develop multi-period expected utility models for portfolio selection. Their contributions have inspired a substantial body of literature that further investigates and analyzes the complexities of portfolio selection.

In practice, a key challenge in implementing portfolio selection models lies in the precise estimation of parameters, especially the expected returns (drifts) of risky assets. Research by \citet{merton1980on} highlights the difficulty of achieving satisfactory accuracy in estimating expected returns, a challenge commonly referred to as the mean-blur problem. \citet{ellsberg1961risk} shows that the investor  (she) is not neutral but averse with respect to the parameter uncertainty, which leads to the concept ``ambiguity" in decision-making. \citet{gilboa1987expected}, \citet{gilboa1989maxmin}, \citet{schmeidler1989subjective} and \citet{yaari1987the} formulate natural axioms that should be satisfied by a preference order to account for ambiguity aversion in the late 1980s, which can be summarized to robust maximin preference. The investor evaluates strategies by maximizing the expected utility based on worst-case beliefs. An important  framework under ambiguity assumes that the unknown expected return $\mu$ lies within a  confidence set $\Lambda=[\mu^{\text{min}},\ \mu^{\text{max}}]$, resulting in the formulation of a maxmin criterion
 \begin{equation}\label{problem1}
\max_\pi\min_{\mu\in\Lambda}\mathbb{E}[U(X^{\pi}(T))],
 \end{equation}
where  \(U\) is the utility function, and \(X^\pi\) denotes the wealth under strategy \(\pi\). Criterion \eqref{problem1} seeks the robust optimal strategy under the worst-case scenario for \(\mu\) within the confidence set \(\Lambda\). This framework has been utilized in studies such as \citet{biagini2017the}, \citet{CF14}, \citet{liang2020robust}, \citet{lin2022robust}, \citet{lin2021optim},  and \citet{sass2022robust}, etc. 

Most studies that analyze Criterion \eqref{problem1} assume a constant confidence set with no learning (cf. \citet{biagini2017the}, \citet{jin2015con}, \citet{liang2020robust}, \citet{lin2022robust}, \citet{lin2021optim}, \citet{sass2022robust}). However, as time increases, learning invloves and additional market information typically reduces the uncertainty over the unknown expected return \(\mu\). \citet{malmendier2011dep} find that the investor’s experienced return has a larger influence on her belief about the drift than risky asset return realizations before birth.  A substantial body of research on ambiguity-neutral investors under uncertain expected returns can be traced back to \citet{gennotte1986optim} and \citet{karatzas1991anote}, with further developments by \citet{bismuth2019portfolio}, \citet{honda2003optim}, \citet{karatzas2001bayesian}, \citet{lakner1995utility,lakner1998optim}, and \citet{rieder2005port}, among others.  A key feature of these studies is that the investor can learn the set of posterior beliefs regarding the unobservable expected return. Typically, an investor begins with a prior distribution over \(\mu\) and updates her beliefs through Bayesian learning. \citet{bismuth2019portfolio} demonstrate that with learning, the uncertainty over \(\mu\) is not invariant; instead, it is affected by the observed log-price of the risky asset and generally decreases over time.  Additionally, as noted by \citet{ES05}, ambiguity can also be reduced over time through the process of learning.

In this paper, we study the robust portfolio selection problem in a maxmin framework with learning about the unknown expected returns. The financial market consists of a risk-free asset and a risky asset, with only asset prices being observable. We formulate the robust optimal problem with learning under time-dependent and state-dependent confidence set $\Lambda_{t,y}=[\mu^{\text{min}}_{t,y},\ \mu^{\text{max}}_{t,y}]$, i.e.,
\begin{equation}\label{problem2}
\max_\pi\min_{\mu\in\Lambda_{t,y}}\mathbb{E}[U(X^{\pi}(T))\ | \ \mathcal{F}_t^S],
 \end{equation}
 where $y$ is an observable state variable representing the best estimation of $\mu$, $\mathcal{F}_t^S$ represents the available  market information (asset prices) before time $t$, $\Lambda_{t,y}$ is a confidence set which depends on the current time $t$ and the observable state $y$. The investor possesses a prior over $\mu$ at the initial time, which we suppose to be Gaussian in this paper. The posterior belief over $\mu$ is updated based on the Bayesian rule and is also Gaussian. The confidence set \(\Lambda_{t,y}\) is constructed based on the posterior belief under a specified confidence level.

To our knowledge, there is limited literature addressing ambiguity and learning within the maxmin framework of continuous-time robust portfolio selection. In the context of learning, the confidence set \(\Lambda_{t,y}\) is both time-dependent and state-dependent.  \citet{campanale2011learn} explores the implications of ambiguity aversion and learning in ambiguous environments, specifically focusing on discrete-time household life-cycle portfolio allocation. \citet{campanale2011learn}'s findings suggest that these factors significantly contribute to explaining the observed patterns in household financial choices. Notably, the model presented by \citet{campanale2011learn} simplifies the risky asset return to two discrete outcomes: high or low. \citet{pei2018life} studies discrete-time life-cycle asset allocation problems with ambiguity aversion and learning about the confidence set. \citet{pei2018life} shows that as agents get older, they learn about the equity premium and increase their allocation to stocks. In this paper, we extend the discrete-time framework into a continuous-time model that integrates both ambiguity and learning. In Criterion \eqref{problem1}, where learning is absent, the worst-case scenario for the expected return \(\mu\) is typically constant, allowing for straightforward application of the verification theorem, see \citet{biagini2017the}, \citet{lin2022robust}, \citet{lin2021optim}, and \citet{sass2022robust}. In Criterion \eqref{problem2},  the confidence set is shortened with time and state-dependent, the worst-case scenario for $\mu$ is also time and state-dependent, which makes it difficult to solve.

The main contributions of this paper are as follows:  First, while most existing literature examines the effects of ambiguity aversion by Criterion \eqref{problem1} with a constant confidence set, few studies address the implications of ambiguity aversion and learning in continuous-time portfolio selection problems under Criterion \eqref{problem2}.  The continuous-time robust portfolio selection problem with learning has been investigated by \citet{branger2013robust} within an entropy penalty framework; however, our approach differs from theirs. We establish the continuous-time robust portfolio selection problem under a time-dependent and state-dependent confidence set for the first time. The confidence set is state-dependent, which causes great challenges in this problem. Second, we derive and solve the HJBI equation associated with the robust portfolio selection problem and show some properties of the related function when $U$ is the CARA utility function. In contrast to most studies based on Criterion \eqref{problem1}, our work does not yield explicit solutions. Instead, the solution to the HJBI equation is represented as a solution to a PDE, specifically a Cauchy problem for a one-dimensional linear second-order parabolic equation with unbounded coefficients. We establish the existence and uniqueness of the PDE and provide estimates for the partial derivatives, which are crucial for the verification theorem. Third, we obtain the robust portfolio selection strategy in a semi-analytical form based on the solution of the HJBI equation and prove the optimality in a rigorous verification theorem under mild conditions. Remarkably, before proving the optimality of the candidate robust optimal solution, we show its admissibility  using sophisticated analytical techniques. In this paper, we reformulate Criterion \eqref{problem2} more rigorously. We show the equivalence of Criterion \eqref{problem2} and the robust problem minimizing over a class of equivalent probability measures. Notably, we show that in our framework, Criterion \eqref{problem2} aligns with the traditional Merton problem when the variance of the prior for \(\mu\) is zero, and corresponds to the optimal problem under partial information when \(\Lambda_{t,y} = \{y\}\).

By incorporating learning and ambiguity aversion, this paper reveals several notable findings. The worst-case scenario for \(\mu\) and the robust optimal investment strategy are determined by comparing the risk-free interest rate \(r\) with the confidence set \(\Lambda_{t,y}\). Specifically, when \(r < \mu^{\text{min}}_{t,y}\) (\(r > \mu^{\text{max}}_{t,y}\)), the worst-case scenario occurs at \(\mu = \mu^{\text{min}}_{t,y}\) (\(\mu = \mu^{\text{max}}_{t,y}\)). Conversely, if \(r \in \Lambda_{t,y}\), any \(\mu \in \Lambda_{t,y}\) represents the worst-case scenario. The robust optimal strategy consists of two components: a myopic demand under the worst-case scenario and a hedging demand. When \(r \in \Lambda_{t,y}\), the myopic demand vanishes, leaving only the hedging demand to address uncertainty in \(\mu\). When \(r < \mu^{\text{min}}_{t,y}\) (\(r > \mu^{\text{max}}_{t,y}\)), the smallest (largest) Sharpe ratio is positive (negative), resulting in a positive (negative) myopic demand under the worst-case scenario. We see that ambiguity aversion results in an adjustment in the myopic demand, leading to a more conservative strategy. Additionally, with learning, this adjustment in myopic demand decreases over time and the investor becomes more aggressive over time, aligning with the findings of \citet{pei2018life}. Besides, in robust optimal problems without learning (see \citet{lin2022robust}), the investment strategy is typically divided into three regions: buying, selling, and non-trading. However, with learning, a hedging demand is introduced (positive when \(y < r\) and negative when \(y > r\)), eliminating the non-trading region. We theoretically analyze and compare the signs of myopic and hedging demands, establishing that the robust optimal strategy can be categorized into three regions: buying, selling, and small trading, as illustrated in Fig.~\ref{f7}. Finally, numerical examples confirm our theoretical results.

The remainder of this paper is organized as follows: Section \ref{Problem formation} sets up the model of the robust portfolio selection problem under state-dependent confidence set. Section \ref{Solution of the HJBI equation} solves the associated HJBI equation. Section \ref{Optimal solution} obtains the robust optimal investment strategy and establishes  the verification theorem. Section \ref{Numerical analysis}  presents and discusses some numerical results and sensitivity analysis. The last section concludes this paper.

\section{Problem formation}
\label{Problem formation}
In this section, we set up the model of the robust portfolio selection problem under a time-dependent and state-dependent confidence set. The confidence set is updated based on Bayesian learning and depends on the observed state of the financial market. We consider an investment problem with one risk-free asset and one risky asset. Let $\left(\Omega, \mathcal{F},\left\{\mathcal{F}_{t}\right\}_{0 \leq t \leq T}, \mathbb{P}\right)$ be a filtered complete probability space satisfying the usual conditions. $T >0$ is a constant,  $[0, T]$ is a finite  time horizon, and the filtration $\left\{\mathcal{F}_{t}\right\}_{0 \leq t \leq T}$ represents the whole information of the financial market. Let $W=\left\{W(t): 0 \leq t \leq T\right\}$ be a standard Brownian motion with respect to (abbr. w.r.t.) filtration $\left\{\mathcal{F}_{t}\right\}_{0 \leq t \leq T}$ under probability measure $\mathbb{P}$. 
\subsection{Financial market}

In the financial market, there is one risk-free asset and one risky asset. The risk-free interest rate is a constant $r$. The risky asset price process $S=\left\{S(t): 0 \leq t \leq T\right\}$  satisfies the following stochastic differential equation (abbr. SDE):
\begin{equation*}
	\begin{aligned}
		\mathrm{d} S(t)=S(t)\left[\mu \mathrm{d} t+\sigma \mathrm{d} W(t)\right],\ t\in [0,T],
	\end{aligned}
\end{equation*} 
where the volatility $\sigma>0$ is a constant and the drift $\mu$ is an unknown constant.  In this context, 
the investor is certain about the volatility of the risky asset's price but uncertain about the drift 
$\mu$. In statistics, while the volatility can be determined with relative confidence, the drift $\mu$ remains difficult to estimate accurately.

The wealth process $X^{\pi}=\left\{X^{\pi}(t): 0 \leq t \leq T\right\}$ with an initial endowment $x_0$ satisfies the following SDE:
\begin{equation*}
	\left\{\begin{array}{l}
		\mathrm{d}X^{\pi}(t)=rX^{\pi}(t)\mathrm{d}t+\pi(t)(\mu-r)\mathrm{d}t+\sigma\pi(t)\mathrm{d}W(t),\ t\in [0,T],\\
		X^{\pi}(0)=x_0.
	\end{array}\right.
\end{equation*}
Here the strategy $\pi=\{\pi(t):0\le t\le T\}$ is a control variable that represents the dollar amount allocated to the risky asset $S$. 

\subsection{State-dependent confidence set}
The investor cannot get the whole information about the drift $\mu$ and the Brownian motion $W$, but she can observe the evolutions of the asset prices. Therefore, the accessible information that the investor knows about the value of the drift $\mu$ is the natural filtration  $\{\mathcal{F}_{t}^{S}\}_{0 \leq t \leq T}$ generated by $S$.

Suppose that $\mu$ is, a Gaussian prior (beliefs of $\mu$ at initial time), independent of the Brownian motion $W$ under probability measure $\mathbb{P}$: 
$$\mu\sim N\left(y_0,\sigma_0^2\right),$$ 
where $y_0\in \mathbb{R}$ and $\sigma_0>0$ are known constants estimated by the investor at the initial time.

Define a process $Y=\left\{Y(t): 0 \leq t \leq T\right\}$  by
$$Y(t) \triangleq \mathbb{E}^{\mathbb{P}}\left[\left.\mu\,\right| \,\mathcal{F}_{t}^{S}\right],\  t \in [0,T].$$
Then, from the perspective of the investor, the process $Y=\left\{Y(t): 0 \leq t \leq T\right\}$ represents the best information about the drift $\mu$ that she can learn from the asset prices. It is noteworthy that $Y$ is the \textit{revealing process} defined in \citet{guan2024equil}, and it plays a vital role in determining the confidence set in this paper.

Let the process $W^{S}=\left\{W^{S}(t): 0 \leq t \leq T\right\}$ be defined by 
$$
W^S(t) \triangleq \int_0^t \frac{\mu-Y(s)}{\sigma} \mathrm{d} s+W(t), \ t \in [0, T].
$$
Based on the Girsanov Theorem (see  \citet[Proposition 2]{bismuth2019portfolio}), $W^{S}=\left\{W^{S}(t): 0 \leq t \leq T\right\}$ is  a standard Brownian motion w.r.t. the filtration $\{\mathcal{F}_{t}^{S}\}_{0 \leq t \leq T}$ under probability measure $\mathbb{P}$. The process $W^{S}$ is called the innovation process in filtering theory.

Obviously, in terms of $W^{S}$, the risky asset price $S$ satisfies the following SDE:
$$\mathrm{d} S(t)=S(t)\left[Y(t) \mathrm{d} t+\sigma \mathrm{d} W^{S}(t)\right], \ t\in [0, T].$$
Let the process $Z=\left\{Z(t): 0 \leq t \leq T\right\}$ be the logarithmic risky asset price process, i.e.,
$$Z(t)\triangleq \log S(t),\  t \in [0,T].$$
Then, based on \citet[Propositions 11 and Remark 5]{bismuth2019portfolio},  the posterior distribution of $\mu$ given $\mathcal{F}_{t}^{S}$ (beliefs of $\mu$ at time $t$) is updated by the Bayesian learning and also  Gaussian distributed:
\begin{equation}\label{equ: posterior}
	\mu|\mathcal{F}_t^S\sim N\left(Y(t),\gamma(t)\right),
\end{equation}
where the conditional variance $\gamma(t)=(\sigma_0^{-2}+t\sigma^{-2})^{-1}$, and the conditional mean $Y(t)$ can be expressed as
$$Y(t)=\gamma(t)\left[\sigma^{-2}\left(Z(t)-Z(0)+\frac{t}{2} \sigma^2\right)+\sigma_0^{-2} y_0\right].$$ 
It is worth noting that the conditional variance $\gamma(t)$ is deterministic and decreases with time $t$,  aligning with the investor's decreasing uncertainty about $\mu$ as more information is gathered.
Besides, the \textit{revealing process} $Y$ satisfies the following SDE:
\begin{equation}\label{equ:y}
	\left\{\begin{array}{l}
		\mathrm{d} Y(t)=\frac{\gamma(t)}{\sigma}\mathrm{d} W^{S}(t)=\gamma(t) \sigma^{-2}\left[(\mu-Y(t))\mathrm{d} t+\sigma \mathrm{d} W(t)\right], \ t\in [0, T],\\
		Y(0)=y_0,
	\end{array}\right.
\end{equation}
which is the form of an Ornstein–Uhlenbeck process. Then under the probability measure $\mathbb{P}$, $Y$ is Gaussian distributed. We can easily obtain  $$\mathbb{E}^{\mathbb{P}}[Y(t)]=y_0,\quad\quad \mathbb{E}^{\mathbb{P}}[Y
^2(t)]=y_0^2+\int_0^t\frac{\gamma^2(s)}{\sigma^2}\mathrm{d}s.$$ Therefore, under probability measure $\mathbb{P}$, $$Y(t)\sim N\left(y_0,\int_0^t\frac{\gamma^2(s)}{\sigma^2}\mathrm{d}s\right),$$ 
i.e., $Y(t)\sim N\left(y_0,\frac{\sigma_0^4t}{\sigma_0^2t+\sigma^2}\right)$  under probability measure $\mathbb{P}$.

The posterior distribution of $\mu$ is given
by \eqref{equ: posterior}.  In contrast to much prior work, we assume that the confidence set of $\mu$ is derived from the posterior distribution of $\mu$. According to \eqref{equ: posterior}, the confidence set for \(\mu\) at time \(t\) under state \(y\) is defined as follows:
$$\Lambda_{t,y}=\left [\mu_{t,y}^{\text{min}}=y-a\sqrt{\gamma(t)},\ \mu_{t,y}^{\text{max}}=y+a\sqrt{\gamma(t)} \right ]\footnote{Here we consider a symmetric confidence interval centered around $y$. Our findings can be easily extended to asymmetric confidence intervals as well.},$$
where $a\ge0$ is a constant. Here $a$ characterizes the confidence level of the confidence set $\Lambda_{t,y}$. As $\mu$ is normally distributed, the confidence level of $\Lambda_{t,y}$ is $2\Phi(a)-1$, where $\Phi(\cdot)$ represents the cumulative distribution function of a standard normal distribution. For $a=2.58, 1.96, 1.645$, the confidence levels of the set $\Lambda_{t,y}$ are $99\%, 95\%, 90\%$, respectively. The investor believes that, with a certain confidence level $2\Phi(a)-1$, the value of $\mu$ at time $t$ lies within the confidence set  $\Lambda_{t,y}$.

\subsection{Robust portfolio selection}
\label{Robust optimal investment problem}
We consider an ambiguity-averse investor who searches for a robust optimal strategy by maximizing the expected utility of the terminal wealth under the worst-case scenario of $\mu$. Then the objective of the ambiguity-averse investor at time $t$ is
\begin{equation} 
	\max_{\pi}\min_{\mu\in\Lambda_{t,Y(t)}}\mathbb{E}[U(X^{\pi}(T))\ |\ \mathcal{F}_t^S],
	\label{obj0}
\end{equation}
where $U(\cdot)$ is a utility function. In the robust optimal problem, unlike previous approaches, the confidence set varies with both the state and time: $\Lambda_{t,Y(t)}$ is determined by the observed state $Y(t)$ and the conditional variance of the posterior distribution of $\mu$. As time elapses, the investor's uncertainty about $\mu$ decreases, resulting in a reduction in the size of the confidence set. Additionally, changes in the financial market lead to variations in $Y$, which in turn affect the confidence set.

\subsection{Reformulation of Problem \eqref{obj0}}
In the following, we reformulate the optimization problem (\ref{obj0}) rigorously. 
Recall that we model an investor who is not sure about the drift $\mu$, but addresses this uncertainty through Bayesian learning using the confidence set $\Lambda_{t,y}$. Then all measurable, $\{\mathcal{F}_{t}^{S}\}_{0\leq t \leq T}$-adapted processes $$\tilde \mu=\left\{\tilde \mu(t):\  \tilde \mu(t) \in \Lambda_{t,Y(t)},\ 0 \leq t \leq T\right\}$$ are possible trajectories for the drift $\mu$. Denote the set of all possible trajectories for the drift $\mu$ by $\mathcal{M}$.

Given $\tilde \mu\in \mathcal{M}$, let process $W^{\tilde \mu}=\left\{W^{\tilde \mu}(t): 0 \leq t \leq T\right\}$  be given by 
$$
W^{\tilde \mu}(t) \triangleq \int_0^t \frac{Y(s)-\tilde \mu(s)}{\sigma} \mathrm{d} s+W^{S}(t), \ t\in [0, T],
$$
and define the probability measure $\mathbb{Q}^{\tilde \mu}$  by 
$$
\left.\frac{\mathrm{d} \mathbb{Q}^{\tilde \mu}}{\mathrm{d}\mathbb{P}}\right|_{\mathcal{F}_T^S} \triangleq \exp \left\{-\int_0^T \frac{Y(t)-\tilde \mu(t)}{\sigma} \mathrm{d} W^{S}(t)-\frac{1}{2}\int_0^T \left(\frac{Y(t)-\tilde \mu(t)}{\sigma} \right)^2\mathrm{d} t \right\}.
$$
Then, by the uniform boundedness of $\left\{\Lambda_{t,y} : 0\leq t \leq T \right\}$ and  Girsanov’s theorem, $W^{\tilde \mu}$ is a standard Brownian motion w.r.t. filtration $\{\mathcal{F}_{t}^{S}\}_{0\leq t \leq T}$ under the probability measure $\mathbb{Q}^{\tilde \mu}$. Besides, under the probability measure  $\mathbb{Q}^{\tilde \mu}$, in terms of $W^{\tilde \mu}$, the risky asset price evolves according to the following SDE:
$$\mathrm{d} S(t)=S(t)\left[\tilde \mu(t) \mathrm{d} t+\sigma \mathrm{d} W^{\tilde \mu}(t)\right], \ t\in [0,T].$$
Define the set of all possible equivalent probability measures by
\[
\mathcal{Q}=\left\{\mathbb{Q}^{\tilde \mu}:\ \tilde \mu\in \mathcal{M}
\right\}.
\]
Therefore, $\min\limits_{\mu\in\Lambda_{t,Y(t)}}\mathbb{E}[U(X^{\pi}(T))\ | \ \mathcal{F}_t^S]$ in (\ref{obj0}) is equivalent to $$\min\limits_{\mathbb{Q}^{\tilde \mu}\in\mathcal{Q}}\mathbb{E}^{\mathbb{Q}^{\tilde \mu}}[U(X^{\pi}(T))\ | \ \mathcal{F}_t^S],$$ i.e., minimizing the expected utility over $\mu$ is equivalent to minimizing it over the equivalent probability measure induced by $\tilde{\mu}$.

Before defining the set of admissible investment strategies, we define the feasible investment strategy. We call an investment process $\pi=\left\{\pi(t): 0 \leq t \leq T\right\}$ feasible if 
\begin{enumerate}
    \item $\pi$ is progressively measurable w.r.t. filtration $\{\mathcal{F}_{t}^{S}\}_{0\leq t \leq T}$.
    \item $\int_0^T\pi^2(t)\mathrm{d} t <\infty,\ \mathbb{Q}^{\tilde \mu}-a.s.,$ for all $\mathbb{Q}^{\tilde \mu}\in\mathcal{Q}$.
\end{enumerate}  We know that $\mathbb{Q}^{\tilde \mu}$ and $\mathbb{P}$ are equivalent, $\forall\ \mathbb{Q}^{\tilde \mu}\in\mathcal{Q}$. Thus, a progressively measurable (relative to $\{\mathcal{F}_{t}^{S}\}_{0\leq t \leq T}$) process $\pi=\left\{\pi(t): 0 \leq t \leq T\right\}$ is feasible if and only if $\int_0^T\pi^2(t)\mathrm{d} t <\infty,\ \mathbb{P}-a.s..$ We denote by $\Pi_0$ the set of all feasible investment strategies. 

Then, given $\tilde \mu\in \mathcal{M}$, the wealth process of the investor with an initial endowment $x_0$ and investment strategy $\pi$ satisfies the following SDE under probability  measure $\mathbb{Q}^{\tilde \mu}$, in terms of $W^{\tilde \mu}$:
\begin{equation*}
	\left\{\begin{array}{l}
		\mathrm{d}X^{\pi}(t)=rX^{\pi}(t)\mathrm{d}t+\pi(t)(\tilde\mu(t)-r)\mathrm{d}t+\sigma\pi(t)\mathrm{d}W^{\tilde \mu}(t),\ t\in [0,T],\\
		X^{\pi}(0)=x_0.
	\end{array}\right.
\end{equation*}
Thus, for any feasible investment process $\pi \in \Pi_0$, the SDE above admits a unique strong solution $X^{\pi}$.
Moreover, the \textit{revealing process} $Y$ satisfies the following SDE under probability  measure $\mathbb{Q}^{\tilde \mu}$, in terms of $W^{\tilde \mu}$:
$$ \mathrm{d} Y(t)=\gamma(t) \sigma^{-2}\left[(\tilde \mu(t)-Y(t))\mathrm{d} t+\sigma \mathrm{d} W^{\tilde \mu}(t)\right],\ t\in [0, T].$$

Let $\Pi \subset \Pi_0$ denote the set of admissible investment strategies. The definition of admissible investment strategy will be given in detail in  Section \ref{Optimal solution}  (see Definition \ref{D4.1}). Thereby, the robust optimal investment problem (\ref{obj0}) at the initial time can be reformulated by searching for the worst-case equivalent probability measure  as follows:  
\begin{equation}
	V(x_0)=\max_{\pi\in\Pi}\min_{\mathbb{Q}^{\tilde \mu}\in\mathcal{Q}}\mathbb{E}^{\mathbb{Q}^{\tilde \mu}}[U(X^{\pi}(T))].
	\label{obj1}
\end{equation}

\section{HJBI equation and related solution}
\label{Solution of the HJBI equation}
In this section, we derive and solve the HJBI equation associated with the robust optimal investment problem \eqref{obj0} and analyze some properties of the corresponding value function.  The solution is expressed in terms of a PDE. In particular, we establish the existence and uniqueness of the PDE solution and provide an estimate for the partial derivative, which plays a key role in the verification theorem. In this system, $\{(X^{\pi}(t), Y(t)):  0\leq t \leq T  \}$ is a Markov process with respect to the filtration $\{\mathcal{F}_{t}^{S}\}_{0 \leq t \leq T}$. Thus, in the optimization problem, the wealth process $X$ and the \textit{revealing process} $Y$ serve as the two state variables. Let the value function at time $t$ be denoted by
$$
V(t,x,y)=\max_{\pi\in\Pi}\min_{\mathbb{Q}^{\tilde \mu}\in\mathcal{Q}}\mathbb{E}^{\mathbb{Q}^{\tilde \mu}}[U(X^{\pi}(T))|X^{\pi}(t)=x,Y(t)=y].
$$
Then we have the following HJBI equation.
\begin{proposition}
    The two-dimensional HJBI equation associated with Problem (\ref{obj0}) is
\begin{equation}
	\left\{\begin{array}{l}
		\sup\limits_{\pi \in \mathbb{R}}\left\{V_t+\frac{1}{2}V_{xx}\pi^2\sigma^2+\frac{1}{2}V_{yy}\frac{\gamma^2(t)}{\sigma^2}+V_{xy}\pi\gamma(t)+V_x r(x-\pi)-V_y\frac{\gamma(t)}{\sigma^2}y\right. \\
		\left.\quad \qquad+\inf\limits_{\mu \in [y-a\sqrt{\gamma(t)},y+a\sqrt{\gamma(t)}]}\{\mu (V_x\pi+V_y\frac{\gamma(t)}{\sigma^2})\}\right\}=0,\\
		V(T,x,y)=U(x).
	\end{array}\right.
	\label{hjbi0}
\end{equation}
\end{proposition}
\begin{proof}
    The derivation is simple and we omit it here.
\end{proof}

In what follows, we provide a solution to the HJBI equation \eqref{hjbi0}. Clearly,
\begin{align*}
	\begin{split}
		\inf_{\mu \in [y-a\sqrt{\gamma(t)},y+a\sqrt{\gamma(t)}]}\{\mu (V_x\pi+V_y\frac{\gamma(t)}{\sigma^2})\}
		&=\begin{cases}
			(y+a\sqrt{\gamma(t)})(V_x\pi\!+\!V_y\frac{\gamma(t)}{\sigma^2}),\ (V_x\pi+V_y\frac{\gamma(t)}{\sigma^2})\leq0,\\
			(y-a\sqrt{\gamma(t)})(V_x\pi\!+\!V_y\frac{\gamma(t)}{\sigma^2}),\  (V_x\pi+V_y\frac{\gamma(t)}{\sigma^2})>0,
		\end{cases}
		\\&=y(V_x\pi+V_y\frac{\gamma(t)}{\sigma^2})-a\sqrt{\gamma(t)}|V_x\pi+V_y\frac{\gamma(t)}{\sigma^2}|.
	\end{split}
\end{align*}
Then the HJBI equation (\ref{hjbi0}) is equivalent to
\begin{equation}
	\left\{\begin{array}{l}
		\sup\limits_{\pi \in \mathbb{R}}\left\{V_t+\frac{1}{2}V_{xx}\pi^2\sigma^2+\frac{1}{2}V_{yy}\frac{\gamma^2(t)}{\sigma^2}+V_{xy}\pi\gamma(t)+V_x r(x-\pi)+yV_x\pi\right. \\
		\left.\quad \qquad  -a\sqrt{\gamma(t)}|V_x\pi+V_y\frac{\gamma(t)}{\sigma^2}|\right\}=0,\\
		V(T,x,y)=U(x).
	\end{array}\right.
	\label{hjbi}
\end{equation}

Suppose that the utility function $U(\cdot)$ is the CARA utility, i.e.,
$$U(x)=-\frac{1}{k}\mathrm{e}^{-kx}, x \in \mathbb{R},$$
where $k>0$ is a constant representing the risk aversion coefficient of the investor.

We guess that a solution $\varphi$  to  \eqref{hjbi} (a candidate value function to Problem \eqref{obj1})   has the following form:
$$\varphi(t,x,y)=-\frac{1}{k}\mathrm{e}^{-k \mathrm{e}^{r(T-t)}x+f(t,y)},$$
where $f(t,y)\in C^{1,2}([0,T]\times\mathbb{R})$.

To solve \eqref{hjbi}, we need to determine the sign of $\varphi_x\pi+\varphi_y\frac{\gamma(t)}{\sigma^2}$. Let $\tilde\pi$ be the zero point of $\varphi_x\pi+\varphi_y\frac{\gamma(t)}{\sigma^2}$, i.e.,

$$\tilde\pi= \frac{f_y\gamma(t)}{k \mathrm{e}^{r(T-t)} \sigma^2}.$$
As $\varphi<0$, we know $\varphi_x=-k\mathrm{e}^{r(T-t)}\varphi>0$. Thus, when $\pi\geq \tilde\pi$, we have $\varphi_x\pi+\varphi_y\frac{\gamma(t)}{\sigma^2}\geq 0$; when $\pi< \tilde\pi$, we have $\varphi_x\pi+\varphi_y\frac{\gamma(t)}{\sigma^2}< 0$.

Let
\begin{equation*}
	\begin{cases}
		\pi_{1}^{*}=-\frac{\varphi_{xy}\gamma(t)-\varphi_x r+\varphi_x(y-a\sqrt{\gamma(t)})}{\varphi_{xx}\sigma^2}= \frac{f_y\gamma(t)-(r-y+a\sqrt{\gamma(t)})}{k \mathrm{e}^{r(T-t)}\sigma^2}=\tilde\pi-\frac{r-y+a\sqrt{\gamma(t)}}{k \mathrm{e}^{r(T-t)} \sigma^2},\\
		\pi_{2}^{*}=-\frac{\varphi_{xy}\gamma(t)-\varphi_x r+\varphi_x(y+a\sqrt{\gamma(t)})}{\varphi_{xx}\sigma^2}= \frac{f_y\gamma(t)-(r-y-a\sqrt{\gamma(t)})}{k \mathrm{e}^{r(T-t)} \sigma^2}=\tilde\pi-\frac{r-y-a\sqrt{\gamma(t)}}{k \mathrm{e}^{r(T-t)} \sigma^2}.
	\end{cases}
\end{equation*}
As $\varphi_{xx}=k^2\mathrm{e}^{2r(T-t)}\varphi<0$, applying the first-order condition, we know that $\pi_{1}^{*}$ is the maximum point of (\ref{hjbi}) when $\varphi_x\pi+\varphi_y\frac{\gamma(t)}{\sigma^2}\geq 0$, and $\pi_{2}^{*}$ is the maximum point of (\ref{hjbi}) when $\varphi_x\pi+\varphi_y\frac{\gamma(t)}{\sigma^2}< 0$.

According to the relationship between $\tilde\pi$, $\pi_{1}^{*}$ and $\pi_{2}^{*}$, we need to distinguish the following three cases based on the relation between $r$ and $\Lambda_{t,y}$. Fig.~\ref{f0} illustrates these three cases by comparing $r$ with $\Lambda_{t,y}$.

 \begin{figure}
 \centering
 {\begin{tikzpicture}
  \draw[very thick,->] (-2.5,0) -- (8,0) node[anchor=north west]{{\large$r$}};
  \draw[very thick,->] (0,-0.25) -- (0,6.5) node[anchor=south east]{{\large$\mu^{*}$}};
  \draw[ultra thick,domain =1:5,dash pattern=on 5pt off 5pt,color=blue] plot(\x,{\x});
  \draw[ultra thick,domain =-2.5:1,color=blue] plot(\x,{1});
  \draw[ultra thick,domain =5:8,color=blue] plot(\x,{5});
  \draw[ultra thick,domain =0:1,loosely dotted,color=blue] plot(\x,{\x});
  \draw[ultra thick,loosely dotted,color=blue] (1,0) -- (1,1);
  \draw[ultra thick,loosely dotted,color=blue] (5,0) -- (5,5);
  \draw[ultra thick,dash pattern=on 5pt off 5pt,color=blue] (1,1) -- (5,1);
  \draw[ultra thick,dash pattern=on 5pt off 5pt,color=blue] (1,5) -- (5,5);
  \draw[ultra thick,loosely dotted,color=blue] (1,1) -- (1,5);
  \fill (1,0)[color=blue]  circle (0.1);
  \node at (1,-0.75) [below] {{\small($y-a\sqrt{\gamma(t)}$)}};
  \node at (1,-0.2) [below] {{\large$\mu_{t,y}^{\text{min}}$}};
  \fill (5,0)[color=blue]  circle (0.1);
  \node at (5,-0.75) [below]{\small{($y+a\sqrt{\gamma(t)}$)}};
  \node at (5,-0.2) [below] {{\large$\mu_{t,y}^{\text{max}}$}};
  \fill[color=blue] (3,0) circle (0.1);
  \node at (3,-0.2) [below] {{\large$y$}};
  \fill (-1,1) circle (0);
  \node at (-1,1) [above] {{\large$\pi_1^{*}$}};
  \fill (3,3) circle (0);
  \node at (3,3) [above] {{\large$\tilde\pi$}};
  \fill (6.5,5) circle (0);
  \node at (6.5,5) [above] {{\large$\pi_2^{*}$}};
\end{tikzpicture}}
\caption{Worst-case scenario for $\mu$ and the corresponding robust optimal strategy.} \label{f0}
\end{figure}
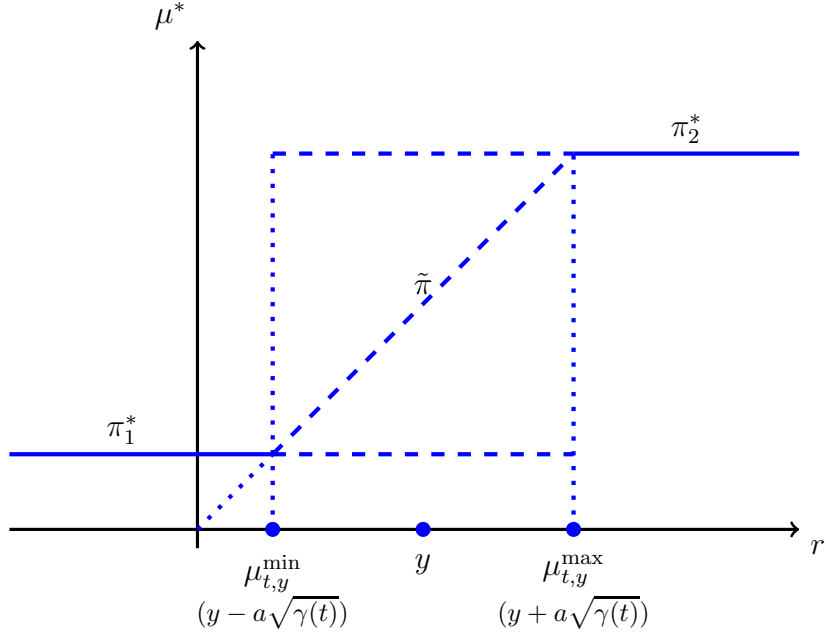

\begin{itemize}
	\item When $ r <\mu_{t,y}^{\text{min}}= y - a\sqrt{\gamma(t)} $, we have that $ \pi_{2}^{*} \geq \pi_{1}^{*} > \tilde{\pi} $. Thus, $ \sup\limits_{\pi \geq \tilde{\pi}} \{\cdot\} $ in \eqref{hjbi} is attained at $ \pi_{1}^{*} $, and $ \sup\limits_{\pi \leq \tilde{\pi}} \{\cdot\} $ in \eqref{hjbi} is attained at $ \tilde{\pi} $. Consequently, the overall supremum $ \sup\limits_{\pi \in \mathbb{R}} \{\cdot\} $ in \eqref{hjbi} is attained at $ \pi_{1}^{*} $.  Substituting $\pi^{*}=\pi_{1}^{*}$ and the expression of $\varphi$ into (\ref{hjbi}),  we obtain the PDE that $f(t,y)$  satisfies
	\begin{equation*}
		\begin{cases}
			f_{t}(t,y)+\frac{1}{2}f_{yy}(t,y)\frac{\gamma^2(t)}{\sigma^2}+f_{y}(t,y)\frac{\gamma(t)(r-y)}{\sigma^2}-\frac{(r-y+a\sqrt{\gamma(t)})^2}{2\sigma^2} =0,\\
			f(T,y)=0 .
		\end{cases}
	\end{equation*}
 
        In this case, the risk-free rate is lower than the minimum value of $\Lambda_{t,y}$, prompting the investor to take a long position in the risky asset. The worst-case scenario for $\mu$ occurs when $\mu^* = \mu_{t,y}^{\text{min}}$. From the expression for $\pi_1^*$, we observe that ambiguity reduces the long position in the risky asset by $\frac{a\sqrt{\gamma(t)}}{k \mathrm{e}^{r(T-t)} \sigma^2}$, resulting in a more conservative strategy.
         
	\item When $r>\mu_{t,y}^{\text{max}}=y+a\sqrt{\gamma(t)}$, we have $\tilde\pi > \pi_{2}^{*} \geq\pi_{1}^{*} $, $\sup\limits_{\pi \geq \tilde \pi}\{\cdot \}$ in \eqref{hjbi} is attained at  $ \tilde\pi$, $\sup\limits_{\pi \leq \tilde \pi}\{\cdot \}$ in \eqref{hjbi} is attained at  $\pi_{2}^{*}$. Consequently, the overall supremum $\sup\limits_{\pi \in \mathbb{R}}\{\cdot \}$ in \eqref{hjbi} is attained at  $ \pi_{2}^{*}$. Substituting $\pi^{*}=\pi_{2}^{*}$ and the expression of $\varphi$ into (\ref{hjbi}),  we obtain the PDE that $f(t,y)$  satisfies
	\begin{equation*}
		\left\{\begin{array}{l}
			f_{t}(t,y)+\frac{1}{2}f_{yy}(t,y)\frac{\gamma^2(t)}{\sigma^2}+f_{y}(t,y)\frac{\gamma(t)(r-y)}{\sigma^2}-\frac{(r-y-a\sqrt{\gamma(t)})^2}{2\sigma^2} =0,\\
			f(T,y)=0.
		\end{array}\right.
	\end{equation*}

In this case, the risk-free rate is higher than the maximum value of $\Lambda_{t,y}$, prompting the investor to take a short position in the risky asset. The worst-case scenario for $\mu$ occurs when $\mu^* = \mu_{t,y}^{\text{max}}$. From the expression for $\pi_2^*$, we observe that ambiguity increases the short position in the risky asset by $\frac{a\sqrt{\gamma(t)}}{k \mathrm{e}^{r(T-t)} \sigma^2}$, also leading to a more conservative strategy.

	\item When ${r}\in\Lambda_{t,y}=[y-a\sqrt{\gamma(t)},y+a\sqrt{\gamma(t)}]$, we have that  $ \pi_{2}^{*} \ge \tilde\pi \ge\pi_{1}^{*} $, $\sup\limits_{\pi \geq \tilde \pi}\{\cdot \}$ in \eqref{hjbi} is attained at  $ \tilde\pi$, $\sup\limits_{\pi \leq \tilde \pi}\{\cdot \}$ in \eqref{hjbi} is attained at  $\tilde\pi $, thus $\sup\limits_{\pi \in \mathbb{R}}\{\cdot \}$ in \eqref{hjbi} is attained at  $ \tilde\pi $. Substituting $\pi^{*}=\tilde\pi$ and the expression of $\varphi$ into (\ref{hjbi}),  we obtain that the PDE that $f(t,y)$  satisfies  is as follows:  
	\begin{equation*}
		\left\{\begin{array}{l}
			f_{t}(t,y)+\frac{1}{2}f_{yy}(t,y)\frac{\gamma^2(t)}{\sigma^2}+f_{y}(t,y)\frac{\gamma(t)(r-y)}{\sigma^2}=0,\\
			f(T,y)=0 .
		\end{array}\right.
	\end{equation*}

In this case, $r$ is within the confidence set $\Lambda_{t,y}$. The investor may take either a short or long position in the stock, depending on her belief about $\mu$. The worst-case scenario for $\mu$ occurs for any $\mu^* \in \Lambda_{t,y}$. Subsequently, the myopic demand diminishes, leaving the investor with solely a hedging demand for the risky asset.
 
\end{itemize}

Summarizing the above statements, we have the following proposition.
\begin{proposition}\label{prop}
	A solution to the HJBI equation (\ref{hjbi0}) is given by 
	$$\varphi(t,x,y)=-\frac{1}{k}\mathrm{e}^{-k \mathrm{e}^{r(T-t)}x+f(t,y)},$$
	where $f(t,y)$ satisfies the following PDE:
	\begin{equation}
		\left\{\begin{array}{l}
			f_{t}(t,y)+\frac{1}{2}f_{yy}(t,y)\frac{\gamma^2(t)}{\sigma^2}+f_{y}(t,y)\frac{\gamma(t)(r-y)}{\sigma^2}\\
			-\left[\frac{(r-y+a\sqrt{\gamma(t)})^2}{2\sigma^2}\mathrm{I}_{\{r-y+a\sqrt{\gamma(t)}\leq 0\}} +\frac{(r-y-a\sqrt{\gamma(t)})^2}{2\sigma^2}\mathrm{I}_{\{r-y-a\sqrt{\gamma(t)}\geq 0\}}\right]=0,\quad (t,y)\in[0,T)\times\mathbb{R},\\
			f(T,y)=0 ,\quad y\in \mathbb{R}.
		\end{array}\right.
		\label{pde1}
	\end{equation}
	
	The worst-case scenario for $\mu$ is given by
 {
	\begin{equation*}
		\mu^*=
		\begin{cases}
			\mu_{t,y}^{\text{min}}, \qquad \quad \qquad {r }< \mu_{t,y}^{\text{min}},\\
   \text{any } \mu\in \Lambda_{t,y} ,\qquad  {r }\in\Lambda_{t,y},\\
			\mu_{t,y}^{\text{max}}, \quad \qquad \qquad {r }> \mu_{t,y}^{\text{max}}.
		\end{cases}
	\end{equation*}
	}
	The suprema of the HJBI equation (\ref{hjbi0}) is given as follows:
	\begin{equation}\label{equ:pihat}
		\begin{aligned}
			\hat\pi(t , y) = \begin{cases} & {\mathrm{e}^{-r(T-t)}\over k\sigma^2}\left[{f_y(t,y)\gamma(t)}+y-{a\sqrt{\gamma(t)}-r}\right]
				,  \qquad {r }< \mu_{t,y}^{\text{min}}, \\&
				{\mathrm{e}^{-r(T-t)}\over k\sigma^2}{f_y(t,y)\gamma(t)}
				,\qquad  \qquad \qquad \qquad\qquad\quad{r }\in\Lambda_{t,y},\\&
				{\mathrm{e}^{-r(T-t)}\over k\sigma^2}\left[{f_y(t,y)\gamma(t)}+y+a\sqrt{\gamma(t)}-r\right]
				,\qquad {r }> \mu_{t,y}^{\text{max}}. \end{cases} \\
		\end{aligned}
	\end{equation}
\end{proposition}

It is important to note that the suprema \(\hat{\pi}(t, y)\) of the HJBI equation (\ref{hjbi0}) are independent of wealth \(x\). The classification of \(\hat{\pi}(t, y)\) into three cases depends on the relationship between \(r\) and the confidence set \(\Lambda_{t,y}\). As shown in \eqref{equ:pihat}, the suprema \(\hat{\pi}(t, y)\) consists of two components: a myopic demand in the worst-case scenario and a hedging demand. When \(r < \mu_{t,y}^{\text{min}}\), the worst-case scenario is given by \(\mu^* =\mu_{t,y}^{\text{min}}\ ( y - a\sqrt{\gamma(t)})\), and the myopic demand is derived by substituting \(\mu^* = y - a\sqrt{\gamma(t)}\) into Merton's portfolio selection problem. Conversely, when \(r > \mu_{t,y}^{\text{max}}\), the worst-case scenario is \(\mu^* = \mu_{t,y}^{\text{max}}\ (y + a\sqrt{\gamma(t)})\), with the myopic demand obtained similarly by substituting \(\mu^* = y + a\sqrt{\gamma(t)}\) into Merton's framework. Finally, when \(r \in \Lambda_{t,y}\), the worst-case scenario can occur for any \(\mu \in \Lambda_{t,y}\), resulting in the investor only exhibiting hedging demand.

The following theorem demonstrates the existence and uniqueness of the solution for the Cauchy problem  (\ref{pde1}) and provides an estimate for the partial derivative. 
\begin{theorem}
	\label{theorem1}
	The Cauchy problem (\ref{pde1}) has a unique solution of class $C^{1,2}([0,T)\times\mathbb{R})\cap C([0,T]\times\mathbb{R})$, which satisfies the polynomial growth condition
	$$\max_{0\leq t \leq T}|f(t,y)|\leq C_1 (1+y^2),\quad \forall y\in \mathbb{R},$$
	for some constant $C_1>0$.
	Moreover, the partial derivative $f_y(t,y)$ satisfies  
	$$\max_{0\leq t \leq T}|f_{y}(t,y)|\leq C_2 (1+|y|),\quad \forall y\in \mathbb{R},$$
	for some constant $C_2>0$.
\end{theorem}

\begin{proof}
	Let $u(t,y)=f(T-t,r-y)$. We obtain the PDE that $u(t,y)$  satisfies 
	\begin{equation}
		\left\{\begin{array}{l}
			u_{t}(t,y)-\frac{1}{2}u_{yy}(t,y)\frac{\gamma^2(T-t)}{\sigma^2}+u_{y}(t,y)\frac{\gamma(T-t)y}{\sigma^2}\\
			=-\left[\frac{(y+a\sqrt{\gamma(T-t)})^2}{2\sigma^2}\mathrm{I}_{\{y+a\sqrt{\gamma(T-t)}\leq 0\}} +\frac{(y-a\sqrt{\gamma(T-t)})^2}{2\sigma^2}\mathrm{I}_{\{y-a\sqrt{\gamma(T-t)}\geq 0\}}\right],\quad (t,y)\in(0,T]\times\mathbb{R},\\
			u(0,y)=0 ,\quad y\in \mathbb{R}.
		\end{array}\right.
		\label{pde2}
	\end{equation}
	\eqref{pde2}  represents a Cauchy problem for a one-dimensional linear second-order parabolic equation with unbounded coefficients. 
	
	Let $$\mathrm{L}(u): = u_{yy}\frac{\gamma^2(T-t)}{2\sigma^2}-u_{y}\frac{\gamma(T-t)y}{\sigma^2}-u_t.$$
	Denote
	$$g(t,y)=\frac{(y+a\sqrt{\gamma(T-t)})^2}{2\sigma^2}\mathrm{I}_{\{y+a\sqrt{\gamma(T-t)}\leq 0\}} +\frac{(y-a\sqrt{\gamma(T-t)})^2}{2\sigma^2}\mathrm{I}_{\{y-a\sqrt{\gamma(T-t)}\geq 0\}},$$
	then $g(t,y)\in C^{0,1} ([0,T]\times\mathbb{R})$.
	
	Because $0<\gamma(T)\leq \gamma(t) \leq \sigma_0^2,\  \forall t \in [0,T]$, the   operator $\mathrm{L}$ is uniformly parabolic in $[0,T]\times \mathbb{R}$. Denote $h(t,y)=M_1\mathrm{e}^{M_2 t}(1+y^2)$, where $M_1>0$, $M_2>0$ are sufficiently large constants. And choose constant $K>0$ such that $$|g(t,y)|\leq Kh(t,y),\ \forall (t,y)\in[0,T]\times\mathbb{R}.$$ We can verify that the assumptions of \citet[Theorems 1 and 3]{besalao1975on}  are satisfied. As such,  there exists a fundamental solution $\Gamma(t,y;\tau,\xi)$ of the parabolic equation $\mathrm{L}(u)=0$, and 
	$$u(t,y)=-\int_0^t \mathrm{d}\tau \int_{\mathbb{R}}\Gamma(t,y;\tau,\xi)g(\tau,\xi)\mathrm{d}\xi$$
	is a solution of class $C^{1,2}((0,T]\times\mathbb{R})\cap C([0,T]\times\mathbb{R})$ for the Cauchy problem (\ref{pde2}). Moreover, we have the pointwise estimate
	$$|u(t,y)|\leq Kth(t,y),\quad \forall (t,y)\in[0,T]\times\mathbb{R}.$$
	Therefore, we have proved that the  Cauchy problem (\ref{pde1}) has a solution $f(t,y)=u(T-t,r-y)$ of class $C^{1,2}([0,T)\times\mathbb{R})\cap C([0,T]\times\mathbb{R})$, and there exists a constant $C_1>0$ such that $$|f(t,y)|=|u(T-t,r-y)|\leq K(T-t)h(T-t,r-y)\leq C_1 (1+y^2),\  \forall (t,y)\in[0,T]\times\mathbb{R}.$$
	
	Next, we  apply the Feynman-Kac formula (see  \citet[Theorems 5.7.6]{karatzas2014brownian}) to obtain the stochastic representation of $f(t,y)$:	\begin{equation}
		f(t,y)=-\mathbb{E}^{\mathbb{P}}\left[\int_t^T\tilde g(s,Y^{t,y}(s))\mathrm{d}s\right],\quad \forall (t,y)\in[0,T]\times\mathbb{R},
		\label{storep}	
	\end{equation}
	where \begin{equation}\label{equ:gt}
		\tilde g(t,y)=g(T-t,r-y),
	\end{equation}and $Y^{t,y}=\{Y^{t,y}(s):t\leq s\leq T\}$ is the unique strong solution of the SDE:
	\begin{equation*}
		\left\{\begin{array}{l}
			\mathrm{d}Y^{t,y}(s)=\frac{\gamma(s)(r-Y^{t,y}(s))}{\sigma^2}\mathrm{d}s+\frac{\gamma(s)}{\sigma}\mathrm{d}W(s),\ s\in [t, T],\\
			Y^{t,y}(t)=y.
		\end{array}\right.
	\end{equation*} 
	In particular, the solution $f(t,y)$ satisfying the polynomial growth condition is unique, and this uniqueness can also be established by using the maximum principle for parabolic equations.
	
	Finally, we complete the proof by presenting the estimate for $f_y$. Using the stochastic representation given in (\ref{storep}), we obtain
	\begin{equation*}
		\begin{aligned}
			&|f(t,y)-f(t,z)|=\left|\mathbb{E}^{\mathbb{P}}\left[\int_t^T\tilde g(s,Y^{t,y}(s))-\tilde g(s,Y^{t,z}(s))\mathrm{d}s\right]\right|\\
			&=\left|\mathbb{E}^{\mathbb{P}}\left[\int_t^T\tilde g_y(s,\eta^{t,z,y}(s))(Y^{t,y}(s)-Y^{t,z}(s))\mathrm{d}s\right]\right|\\
			&\leq \mathbb{E}^{\mathbb{P}}\left[\int_t^T|\tilde g_y(s,\eta^{t,z,y}(s))||Y^{t,y}(s)-Y^{t,z}(s)|\mathrm{d}s\right]\\
			&\leq \mathbb{E}^{\mathbb{P}}\left[\int_t^TN_1\left(1+|Y^{t,y}(s)|+|Y^{t,z}(s)|\right)|Y^{t,y}(s)-Y^{t,z}(s)|\mathrm{d}s\right]\\
			&\leq TN_1\mathbb{E}^{\mathbb{P}}\left[\left(1+\max_{t\leq s\leq T}|Y^{t,y}(s)|+\max_{t\leq s\leq T}|Y^{t,z}(s)|\right)\max_{t\leq s\leq T}\left(|Y^{t,y}(s)-Y^{t,z}(s)|\right)\right]\\
			&\leq N_2 \left\{\mathbb{E}^{\mathbb{P}}\left[\left(1+\max_{t\leq s\leq T}|Y^{t,y}(s)|+\max_{t\leq s\leq T}|Y^{t,z}(s)|\right)^2\right]\mathbb{E}^{\mathbb{P}}\left[\max_{t\leq s\leq T}\left(|Y^{t,y}(s)-Y^{t,z}(s)|\right)^2\right]\right\}^{\frac{1}{2}}\\
			&\leq N_3 \left\{\mathbb{E}^{\mathbb{P}}\left[\left(1+\max_{t\leq s\leq T}|Y^{t,y}(s)|^2+\max_{t\leq s\leq T}|Y^{t,z}(s)|^2\right)\right]\mathbb{E}^{\mathbb{P}}\left[\max_{t\leq s\leq T}\left(|Y^{t,y}(s)-Y^{t,z}(s)|\right)^2\right]\right\}^{\frac{1}{2}}\\
			&\leq N_4 \left[\left(1+|y|^2+|z|^2\right)|y-z|^2\right]^{\frac{1}{2}}\\
			&\leq N_5\left(1+|y|\vee|z|\right)|y-z|,
		\end{aligned}
	\end{equation*}
	where $\eta^{t,z,y}(s)$ is between $Y^{t,y}(s)$ and $Y^{t,z}(s)$, $\forall t\in [0,T]$, $y,z \in \mathbb{R}$, and $N_i>0, \ i=1,\cdots ,5, $ are some positive constants. The last but one inequality holds based on \citet[Theorem 1.6.3]{yong1999sto}.
 
	Thus, we obtain
	$$\max_{0\leq t \leq T}|f_{y}(t,y)|\leq C_2 (1+|y|),\quad \forall\ y\in \mathbb{R},$$
	for some constant $C_2>0$.
\end{proof}

According to the above discussions, we have that $\varphi(t,x,y)$ is a solution of the HJBI equation (\ref{hjbi0}). The candidate robust optimal strategy and worst-case scenario are also given in Proposition \ref{prop}. Next, we analyze the signs of $f(t,y)$, $f_y(t,y)$ and the candidate robust optimal strategy $\hat{\pi}(t,y)$.

\begin{theorem}\label{thm:ffy}
The solution $f(t,y)$ to the Cauchy problem (\ref{pde1}) satisfies
$$f(t,y)<0, \ \forall\   t\in[0,T).$$
Besides, $f_y(t,y)$ is monotonically decreasing with respect to $y$ when $y\in \mathbb{R},\ t\in[0,T)$ and satisfies
\begin{equation}\label{equ: fy}
		\begin{aligned}
			f_y(t , y)  \begin{cases}  >0 ,\qquad y<r, \ t\in[0,T) , \\
				=0, \qquad  y=r, \ t\in[0,T], \quad\text{or} \quad y\in \mathbb{R}, \ t=T,\\
				 <0 ,\qquad  y>r, \ t\in[0,T) . \end{cases} \\
		\end{aligned}
	\end{equation}
\end{theorem}
\begin{proof}
Using the stochastic representation given in (\ref{storep}), we have 
 \begin{equation*}
		f(t,y)=-\mathbb{E}^{\mathbb{P}}\left[\int_t^T\tilde g(s,Y^{t,y}(s))\mathrm{d}s\right]\leq 0,\quad \forall\  (t,y)\in[0,T]\times\mathbb{R},
	\end{equation*}
	where  $$Y^{t,y}(s)=\frac{\gamma(s)y}{\gamma(t)}+\frac{\gamma(s)r}{\sigma^2}(s-t)+\frac{\gamma(s)}{\sigma}\left(W(s)-W(t)\right), \ s \in [t,T]. $$
	Therefore $$Y^{t,y}(s)\sim N\left(\frac{\gamma(s)y}{\gamma(t)}+\frac{\gamma(s)r}{\sigma^2}(s-t),\frac{\gamma^2(s)}{\sigma^2}(s-t)\right),\ s \in [t,T].$$
As \( \text{Var}(Y^{t,y}(s)) = \frac{\gamma^2(s)}{\sigma^2}(s-t) > 0 \) for \( s \in (t, T] \), we have \( \mathbb{P}\left(\{\tilde{g}(s, Y^{t,y}(s)) > 0\}\right) > 0 \) for \( \forall\  t\in[0,T)\) and \( s \in (t, T] \). Then
	$$		\mathbb{E}^{\mathbb{P}}\left[\tilde g(s,Y^{t,y}(s))\right]>0,\ s \in (t,T].
	$$ 
Thus
	\begin{equation*}
	\begin{aligned}
		f(t,y)&=-\mathbb{E}^{\mathbb{P}}\left[\int_t^T\tilde g(s,Y^{t,y}(s))\mathrm{d}s\right]\\
		&=-\int_t^T\mathbb{E}^{\mathbb{P}}\left[\tilde g(s,Y^{t,y}(s))\right]\mathrm{d}s<0,\quad t\in[0,T).
	\end{aligned}
	\end{equation*}
 Furthermore, when $a>0$, we can show that $$f_y(t,y)\equiv 0,\quad  r\in\Lambda_{t,y}\quad t\in[0,T)$$ is not true. Otherwise, we know $$f_{yy}(t,y)\equiv 0,\quad  r\in\Lambda_{t,y},\quad  t\in[0,T).$$ As such, from the PDE \eqref{pde1} that $f(t,y)$ satisfies, we know $$f_{t}(t,y)\equiv 0,\quad r\in\Lambda_{t,y},\quad  t\in[0,T).$$ Thus, combining with the terminal condition, we have $$f(t,y)\equiv 0,\quad  r\in\Lambda_{t,y},\quad t\in[0,T),$$  which leads to a contradiction.
 
 Moreover, we can provide a more precise estimate for $f_y(t,y)$. Using the stochastic representation given in (\ref{storep}) and the dominated convergence theorem for derivatives, we have 
 \begin{equation*}
 \begin{aligned}
		f_y(t,y)&=-\mathbb{E}^{\mathbb{P}}\left[\int_t^T\tilde g_y(s,Y^{t,y}(s))\frac{\mathrm{d}Y^{t,y}(s)}{\mathrm{d}y}\mathrm{d}s\right]\\
  &=-\int_t^T\mathbb{E}^{\mathbb{P}}\left[\tilde g_y(s,Y^{t,y}(s))\frac{\gamma(s)}{\gamma(t)}\right]\mathrm{d}s,\quad \forall (t,y)\in[0,T]\times\mathbb{R},
\end{aligned}
\end{equation*}
 where $\tilde g_y(s,y)=-\frac{r-y+a\sqrt{\gamma(s)}}{\sigma^2}\mathrm{I}_{\{r-y+a\sqrt{\gamma(s)}\leq 0\}}-\frac{r-y-a\sqrt{\gamma(s)}}{\sigma^2}\mathrm{I}_{\{r-y-a\sqrt{\gamma(s)}\geq 0\}}=- \tilde g_2(s,y)$,  $\tilde g_2$ is defined in \eqref{equ:g2t}. 

 Define $$p(s;t,y)=\mathbb{E}^{\mathbb{P}}\left[Y^{t,y}(s)\right]=\frac{\gamma(s)y}{\gamma(t)}+\frac{\gamma(s)r}{\sigma^2}(s-t), \ s \in [t,T].$$
 We need to distinguish the following three
cases based on the relation between $y$ and $r$.
\begin{itemize}
	\item When $y=r$, we have   $p(s;t,y)=\frac{\gamma(s)r}{\gamma(t)}+\frac{\gamma(s)r}{\sigma^2}(s-t)\equiv r, \ s \in [t,T]$. Thus, $\mathbb{E}^{\mathbb{P}}\left[\tilde g_y(s,Y^{t,y}(s))\frac{\gamma(s)}{\gamma(t)}\right]=0, \ s \in [t,T].$ Then $f_y(t,y)=0,\ t\in[0,T]$.
  \item When $y<r$, we have   $p(s;t,y)-r=\frac{\gamma(s)}{\gamma(t)}(y-r)$, as such $p(s;t,y)=\frac{\gamma(s)y}{\gamma(t)}+\frac{\gamma(s)r}{\sigma^2}(s-t)< r, \ s \in [t,T]$. Thus, $\mathbb{E}^{\mathbb{P}}\left[\tilde g_y(s,Y^{t,y}(s))\frac{\gamma(s)}{\gamma(t)}\right]<0, \ s \in (t,T].$ Then $f_y(t,y)>0,\ t\in[0,T)$. And $f_y(t,y)$ is monotonically decreasing with respect to $y$ when $y<r,\ t\in[0,T)$.
  \item When $y>r$, we have   $p(s;t,y)-r=\frac{\gamma(s)}{\gamma(t)}(y-r)$, as such $p(s;t,y)=\frac{\gamma(s)y}{\gamma(t)}+\frac{\gamma(s)r}{\sigma^2}(s-t)> r, \ s \in [t,T]$. Thus, $\mathbb{E}^{\mathbb{P}}\left[\tilde g_y(s,Y^{t,y}(s))\frac{\gamma(s)}{\gamma(t)}\right]>0, \ s \in (t,T].$ Then, $f_y(t,y)<0,\ t\in[0,T)$. And $f_y(t,y)$ is monotonically decreasing with respect to $y$ when $y>r,\ t\in[0,T)$.
\end{itemize}

Therefore, $f_y(t,y)$ is monotonically decreasing with respect to $y$ when $y\in \mathbb{R},\ t\in[0,T)$.
Thus the theorem holds.
\end{proof}

Theorem \ref{thm:ffy} demonstrates that the hedging demand in \eqref{equ:pihat} is positive when \(y < r\) and negative when \(y > r\), aligning with results in portfolio selection under partial information. According to \eqref{equ:pihat}, the myopic demand is positive when \(y > a\sqrt{\gamma(t)} + r\) and negative when \(y < -a\sqrt{\gamma(t)} + r\). In the interval \(y \in [-a\sqrt{\gamma(t)} + r, a\sqrt{\gamma(t)} + r]\), the myopic demand is zero, and the sign of \(\hat{\pi}(t,y)\) is determined by the hedging demand.  The sign of \(\hat{\pi}(t,y)\) for $y$ outside $[-a\sqrt{\gamma(t)} + r, a\sqrt{\gamma(t)} + r]$ is influenced by both myopic and hedging demands. Combining \eqref{equ:pihat} and \eqref{equ: fy}, we can more clearly ascertain the sign of \(\hat{\pi}(t,y)\).
 \begin{figure}[htbp]
 \centering
 {\begin{tikzpicture}
  \draw[very thick,->] (-2,0) -- (8,0) node[anchor=north west]{{\large$y$}}{};
  \draw[very thick,->] (0,-2.5) -- (0,3) node[anchor=south east]{{\large$\hat\pi(t,y)$}};
  \draw[ultra thick,color=blue] (-2,-2.5) -- (1,0.5) ;
  \draw[ultra thick,color=blue] (5,-0.5) -- (8,2.5) ;
  \draw[ultra thick,color=blue] (1,0.5) -- (5,-0.5) ;
  \fill (1,0)[color=blue]  circle (0.1);
  \node at (1,0) [below] {{\large$r-a\sqrt{\gamma(t)}$}};
  \fill (5,0)[color=blue]  circle (0.1);
  \node at (5,0) [below] {{\large$r+a\sqrt{\gamma(t)}$}};
  \fill[color=blue] (3,0) circle (0.1);
  \node at (3,-0.25) [below] {{\large$r$}};
  \node at (-0.25,0.25) {{\large$0$}};
  \put(-22,-26){\color{red}{$-$}}
  \put(22,22){\color{red}{$+$}}
  \put(135,-26){\color{red}{$-$}}
   \put(187,22){\color{red}{$+$}}
 \end{tikzpicture}}
\caption{The sign of the robust optimal
feedback function $\hat\pi(t,y)$.} \label{f7}
\end{figure}
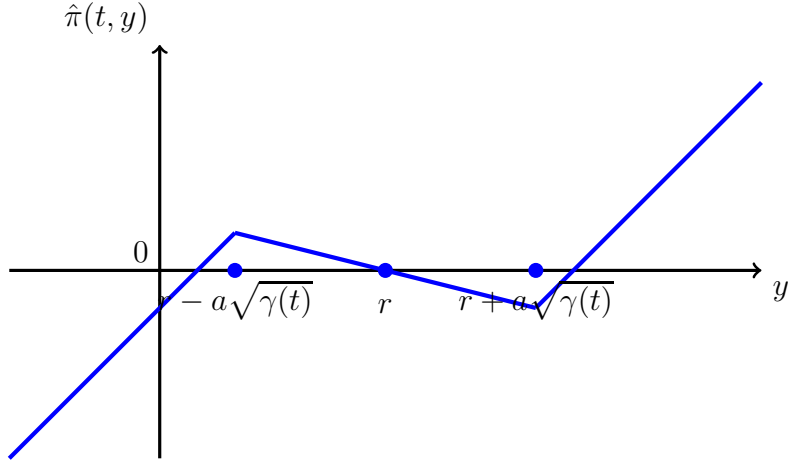
\begin{corollary}\label{coro:hatpi}
   For fixed $a>0, \ 0\leq t <T$, the sign of the robust optimal
feedback function $\hat\pi(t,y)$ is illustrated in Fig.~\ref{f7}.
\end{corollary}
\begin{proof}
The sign of \(\hat{\pi}(t,y)\) when \(y \in \left[r - a\sqrt{\gamma(t)}, r + a\sqrt{\gamma(t)}\right]\) follows directly by combining \eqref{equ:pihat} and \eqref{equ: fy}. Next, we will show the signs of \(\hat{\pi}(t,y)\) when \(y\) is relatively small or large.

Let \( f|_{a=0}(t,y) \) denote the solution to \eqref{pde1} when \( a=0 \), which is given by \eqref{equ:fa=0}. Fixing $ t\in[0,T), a>0$, we have
\begin{equation*}
	\begin{aligned}
		\mathbb{E}^{\mathbb{P}}\left[\tilde g(s,Y^{t,y}(s))\right]<\mathbb{E}^{\mathbb{P}}\left[\tilde g|_{a=0}(s,Y^{t,y}(s))\right],\ s \in (t,T].
	\end{aligned}
	\end{equation*}
Then
\begin{equation*}
	\begin{aligned}
		0>f(t,y)&=-\int_t^T\mathbb{E}^{\mathbb{P}}\left[\tilde g(s,Y^{t,y}(s))\right]\mathrm{d}s\\
  &>-\int_t^T\mathbb{E}^{\mathbb{P}}\left[\tilde g|_{a=0}(s,Y^{t,y}(s))\right]\mathrm{d}s\\
  &=f|_{a=0}(t,y)=f_1(t)y^2-2rf_1(t)y+f_3(t).
	\end{aligned}
	\end{equation*}
Moreover, when $y>r$, 
\begin{equation*}
	\begin{aligned}
		\mathbb{E}^{\mathbb{P}}\left[\tilde g_y(s,Y^{t,y}(s))\right]<\mathbb{E}^{\mathbb{P}}\left[\tilde g_y|_{a=0}(s,Y^{t,y}(s))\right],\ s \in [t,T].
	\end{aligned}
	\end{equation*}
Therefore,
\begin{equation*}
	\begin{aligned}
		0>f_y(t,y)&=-\int_t^T\mathbb{E}^{\mathbb{P}}\left[\tilde g_y(s,Y^{t,y}(s))\frac{\gamma(s)}{\gamma(t)}\right]\mathrm{d}s,\\
  &>-\int_t^T\mathbb{E}^{\mathbb{P}}\left[\tilde g_y|_{a=0}(s,Y^{t,y}(s))\frac{\gamma(s)}{\gamma(t)}\right]\mathrm{d}s,\\
  &=f_y|_{a=0}(t,y)=2f_1(t)y-2rf_1(t)=\frac{\gamma(T)-\gamma(t)}{\gamma^2(t)}(y-r).
	\end{aligned}
	\end{equation*}
\begin{equation}\label{equ:fy1}
f_y(t,y)\gamma(t)-r+y-a\sqrt{\gamma(t)}> \frac{\gamma(T)}{\gamma(t)}(y-r)-a\sqrt{\gamma(t)}.
\end{equation}
Combining with  \eqref{equ:pihat}, we find that, when \(y\) is relatively large, \(\hat{\pi}(t,y) > 0\).

Similarly, when $y<r$,
\begin{equation*}
	\begin{aligned}
		0<f_y(t,y)&=-\int_t^T\mathbb{E}^{\mathbb{P}}\left[\tilde g_y(s,Y^{t,y}(s))\frac{\gamma(s)}{\gamma(t)}\right]\mathrm{d}s,\\
  &<-\int_t^T\mathbb{E}^{\mathbb{P}}\left[\tilde g_y|_{a=0}(s,Y^{t,y}(s))\frac{\gamma(s)}{\gamma(t)}\right]\mathrm{d}s,\\
  &=f_y|_{a=0}(t,y)=2f_1(t)y-2rf_1(t)=\frac{\gamma(T)-\gamma(t)}{\gamma^2(t)}(y-r).
	\end{aligned}
	\end{equation*}
\begin{equation}\label{equ:fy2}
f_y(t,y)\gamma(t)-r+y+a\sqrt{\gamma(t)}< \frac{\gamma(T)}{\gamma(t)}(y-r)+a\sqrt{\gamma(t)}.
\end{equation}
Therefore,  combining with  \eqref{equ:pihat}, we conclude that, when \(y\) is relatively small, \(\hat{\pi}(t,y) < 0\).
\end{proof}

Corollary \ref{coro:hatpi} indicates that when $y \in (r - a\sqrt{\gamma(t)}, r+a\sqrt{\gamma(t)})$, the myopic demand disappears, and the sign of the $\hat{\pi}(t,y)$ depends on the sign of the hedging demand. Consequently, the investor adopts a long position in the risky asset when $y \in (r - a\sqrt{\gamma(t)}, r)$, whereas when $y \in (r, r + a\sqrt{\gamma(t)})$, the investor takes a short position.

However, when \(y < r - a\sqrt{\gamma(t)}\), the investment strategy comprises a negative myopic demand and a positive hedging demand. Near \(y = r - a\sqrt{\gamma(t)}\), the myopic demand approaches zero, allowing the hedging demand to dominate and resulting in a positive optimal feedback function. Conversely, when \(y\) is relatively small, the myopic demand prevails, leading to a negative robust optimal feedback function. For \(y > r + a\sqrt{\gamma(t)}\), if \(y\) is close to \(r + a\sqrt{\gamma(t)}\), the hedging demand dominates, yielding a negative robust optimal feedback function. However, as \(y\) increases further, the myopic demand takes precedence, resulting in a positive feedback function.

\section{Optimal solution}
\label{Optimal solution}
In this section, we solve the robust optimal investment problem (\ref{obj1}) in Subsection \ref{Robust optimal investment problem} based on the solution of the HJBI equation (\ref{hjbi0}).  We verify that the candidate value function $\varphi$  and related suprema $\hat{\pi}$  given in Proposition \ref{prop} solve Problem (\ref{obj1}).  First, we present the definition of admissible
investment strategy. Before proving the optimality of the candidate robust optimal solution, we show the admissibility of $\hat{\pi}$. Finally, we conclude this section with the verification theorem.

\subsection{Definition of admissible investment strategy}
Recall in Subsection \ref{Robust optimal investment problem}, we do not give a detailed definition of admissible investment strategy. Now, we define the admissible investment strategy based on the solution $\varphi(t,x,y)$ of the HJBI equation (\ref{hjbi0}).
\begin{definition}[admissible investment strategy] \label{D4.1}An investment strategy $\pi$ is said to be admissible if the following conditions are satisfied:
	
	(i) $\pi \in \Pi_0$.
	
	(ii) \begin{equation*}
		\left\{\int_0^t\left( \varphi_x\left(s, X^{\pi}(s), Y(s)\right)\sigma \pi(s)+\varphi_y\left(s, X^{\pi}(s), Y(s)\right) \frac{\gamma(s)}{\sigma}\right)\mathrm{d} W^{\tilde  \mu}(s)\right\}_{0\le t\le T} 
	\end{equation*}
	is a supermartingale w.r.t. filtration $\{\mathcal{F}_{t}^{S}\}_{0\leq t \leq T}$ under probability measure $\mathbb{Q}^{\tilde  \mu}$ for any $\mathbb{Q}^{\tilde  \mu}\in \mathcal{Q}$.
	\label{def}
\end{definition}

We denote the set of admissible strategies by $\Pi$. 

\subsection{Admissibility of the optimal investment strategy}
In this subsection, we first obtain the robust optimal investment strategy based on the solution $\varphi(t,x,y)$ and the suprema $\hat\pi(t , y)$ of the HJBI  equation (\ref{hjbi0}). Then, we prove the admissibility of the robust optimal solution. The proof of optimality is left in the last subsection.

According to the solution $\varphi(t,x,y)$ and the suprema $\hat\pi(t,y)$ of the HJBI  equation (\ref{hjbi0}), $$\pi^{*}\triangleq\{\pi^{*}(t)= \hat\pi(t , Y(t)):\ 0\leq t \leq T\}$$ is a candidate robust optimal investment strategy. In particular, it is noteworthy that $\pi^{*}$ is independent of the wealth process $X$.

Using Theorem \ref{theorem1}, we have  $$|\hat\pi(t , y)|\leq C_3(1+|y|),\quad t\in [0,T],\quad y\in \mathbb{R},$$  
where $C_3>0$ is a constant. Thus  $\pi^{*}$ is feasible because
\begin{equation*}
	\begin{aligned}
		&\mathbb{E}^{\mathbb{P}}\left[\int_0^T(\pi^{*}(t))^2\mathrm{d}t\right]\leq C_4 \mathbb{E}^{\mathbb{P}}\left[\int_0^T(1+Y^2(t))\mathrm{d}t\right]=C_4 \int_0^T\mathbb{E}^{\mathbb{P}}\left[1+Y^2(t)\right]\mathrm{d}t\\
		&=C_4 \int_0^T \left(1+y^2_0+\int_0^t\frac{\gamma^2(s)}{\sigma^2}\mathrm{d}s\right)\mathrm{d}t\leq C_5<+\infty,
	\end{aligned}
\end{equation*}
where  $C_4 >0 $ and  $ C_5>0$ are constants.  To establish the admissibility of $\pi^{*}$, we first present two lemmas.
\begin{lemma}\label{lemma:w}
If $2Ct^2\frac{\sigma_0^4}{\sigma^2}<1$, then
$$\mathbb{E}^{\mathbb{P}}\left[\mathrm{exp}\left\{C\int_0^t\gamma^2(s) \frac{ W^2(s)}{\sigma^2}\mathrm{d} s\right\}\right]<+\infty.$$
\end{lemma}
\begin{proof}
Using Taylor's expansion and Hölder's inequality, we have 
\begin{equation*}
		\begin{aligned}
			&\mathbb{E}^{\mathbb{P}}\left[\mathrm{exp}\left\{C\int_0^t\gamma^2(s) \frac{ W^2(s)}{\sigma^2}\mathrm{d} s\right\}\right]
			=1+\sum_{n=1}^{+\infty}\frac{C^n}{n!}\mathbb{E}^{\mathbb{P}}\left[\left(\int_0^t\gamma^2(s) \frac{ W^2(s)}{\sigma^2}\mathrm{d} s\right)^n\right]\\
			\leq & 1+\sum_{n=1}^{+\infty}\frac{C^n}{n!}\mathbb{E}^{\mathbb{P}}\left[t^{n-1}\left(\frac{\sigma_0^4}{\sigma^2}\right)^{n}\int_0^tW^{2n}(s)\mathrm{d} s\right]
			= 1+\sum_{n=1}^{+\infty}\frac{C^n}{n!}t^{n-1}\left(\frac{\sigma_0^4}{\sigma^2}\right)^{n}\int_0^t\mathbb{E}^{\mathbb{P}}\left[W^{2n}(s)\right]\mathrm{d} s\\
			= & 1+\sum_{n=1}^{+\infty}\frac{C^n}{n!}t^{n-1}\left(\frac{\sigma_0^4}{\sigma^2}\right)^{n}\int_0^t\frac{(2n)!}{2^n n!}s^n\mathrm{d} s
			= 1+\sum_{n=1}^{+\infty}\frac{C^n}{n!}t^{n-1}\left(\frac{\sigma_0^4}{\sigma^2}\right)^{n}\frac{(2n)!}{2^n n!}\frac{t^{n+1}}{n+1}\\
			\leq & 1+\sum_{n=1}^{+\infty}\left(2Ct^2\frac{\sigma_0^4}{\sigma^2}\right)^n
			=\frac{1}{1-\left(2Ct^2\frac{\sigma_0^4}{\sigma^2}\right)},
		\end{aligned}
	\end{equation*}
	if $2Ct^2\frac{\sigma_0^4}{\sigma^2}<1$.  Thus, the lemma holds.
\end{proof}

\begin{lemma}\label{lemma:y2}
 Suppose that there exist constants $\epsilon_1, \epsilon_2, \epsilon_3>1$ such that 
 \begin{align}
     &2C(1+{\epsilon_2}+\frac{1}{\epsilon_3})T^2{\sigma_0^4\over \sigma^2}<1,\label{condition1}\\
&2C(1+\epsilon_1+\epsilon_3)(T\sigma_0^2-2\sigma^2\ln \frac{\sigma_0^2T+\sigma^2}{\sigma^2}-\frac{\sigma^4}{\sigma_0^2T+\sigma^2}+\sigma^2)<1,
         (\text{or } 2C(1+\epsilon_1+\epsilon_3)\frac{\sigma_0^6T^3}{3\sigma^4}<1),\label{condition2} \end{align}
 then 
 $$\mathbb{E}^{\mathbb{P}}\left[\mathrm{exp}\left\{C\int_0^tY^2(s)\mathrm{d} s\right\}\right]<+\infty, \ t\in[0,T].$$
\end{lemma}
\begin{proof}
	First,  by \eqref{equ:y}, we have 
	\begin{equation*}
		\begin{aligned}
			Y(t)&=\gamma(t)\left(\sigma^{-2}\left(Z(t)-Z(0)+\frac{t}{2} \sigma^2\right)+\sigma_0^{-2} y_0\right)=\gamma(t)\left(\frac{\mu t}{\sigma^2}+ \frac{ W(t)}{\sigma}+\frac{y_0}{\sigma_0^2}\right).
		\end{aligned}
	\end{equation*}
	Thus, for any $\epsilon_1, \epsilon_2, \epsilon_3 >0$, we have
	\begin{equation*}
		\begin{aligned}
			Y^2(t)&=\gamma^2(t)\left(\frac{\mu t}{\sigma^2}+ \frac{ W(t)}{\sigma}+\frac{y_0}{\sigma_0^2}\right)^2\\
			&\leq \gamma^2(t)\left((1+\epsilon_1+\epsilon_3)\frac{\mu ^2t^2}{\sigma^4}+ (1+\epsilon_2+\frac{1}{\epsilon_3})\frac{ W^2(t)}{\sigma^2}+(1+\frac{1}{\epsilon_1}+\frac{1}{\epsilon_2})\frac{y_0^2}{\sigma_0^4}\right).
		\end{aligned}
	\end{equation*}

	As $\mu$ is independent of the Brownian motion $W$ under probability measure $\mathbb{P}$, we have 
	\begin{equation}\label{equ:yestimate}
		\begin{aligned}
	&\mathbb{E}^{\mathbb{P}}\left[\mathrm{exp}\left\{C\int_0^tY^2(s)\mathrm{d} s\right\}\right]\\
			\leq & \mathbb{E}^{\mathbb{P}}\left[\mathrm{exp}\left\{C\int_0^t\gamma^2(s)\left((1+\epsilon_1+\epsilon_3)\frac{\mu ^2s^2}{\sigma^4}+ (1+\epsilon_2+\frac{1}{\epsilon_3})\frac{ W^2(s)}{\sigma^2}+(1+\frac{1}{\epsilon_1}+\frac{1}{\epsilon_2})\frac{y_0^2}{\sigma_0^4}\right)\mathrm{d} s\right\}\right]\\
			=&\mathrm{e}^{C\int_0^t(1+\frac{1}{\epsilon_1}+\frac{1}{\epsilon_2})\gamma^2(s)\frac{y_0^2}{\sigma_0^4}\mathrm{d} s}\mathbb{E}^{\mathbb{P}}\left[\mathrm{exp}\left\{C\int_0^t\gamma^2(s)\left((1+\epsilon_1+\epsilon_3)\frac{\mu ^2s^2}{\sigma^4}+ (1+\epsilon_2+\frac{1}{\epsilon_3})\frac{ W^2(s)}{\sigma^2}\right)\mathrm{d} s\right\}\right]\\
			=& C_{t,C,\epsilon_1, \epsilon_2} \mathbb{E}^{\mathbb{P}}\left[\mathrm{e}^{C(1+\epsilon_1+\epsilon_3)\int_0^t\gamma^2(s)\frac{\mu ^2s^2}{\sigma^4}\mathrm{d} s}\right]\mathbb{E}^{\mathbb{P}}\left[\mathrm{exp}\left\{C(1+\epsilon_2+\frac{1}{\epsilon_3})\int_0^t\gamma^2(s) \frac{ W^2(s)}{\sigma^2}\mathrm{d} s\right\}\right]\\
			=&C_{t,C,\epsilon_1, \epsilon_2} \mathbb{E}^{\mathbb{P}}\left[\mathrm{e}^{C_{t,C,\epsilon_1, \epsilon_3} \mu ^2}\right]\mathbb{E}^{\mathbb{P}}\left[\mathrm{exp}\left\{C(1+\epsilon_2+\frac{1}{\epsilon_3})\int_0^t\gamma^2(s) \frac{ W^2(s)}{\sigma^2}\mathrm{d} s\right\}\right],
		\end{aligned}
	\end{equation}
	where $$C_{t,C,\epsilon_1, \epsilon_2}=\mathrm{exp}\left\{C\int_0^t(1+\frac{1}{\epsilon_1}+\frac{1}{\epsilon_2})\gamma^2(s)\frac{y_0^2}{\sigma_0^4}\mathrm{d} s\right\}<\infty,\  \forall t\in [0,T],\ C,\
  \epsilon_1,\  \epsilon_2 >0, $$ and 
	$C_{t,C,\epsilon_1, \epsilon_3} =C(1+\epsilon_1+\epsilon_3)\int_0^t\gamma^2(s)\frac{s^2}{\sigma^4}\mathrm{d} s$.\\
 Furthermore, we have 
	\begin{equation*}
		\begin{aligned}
			C_{t,C,\epsilon_1, \epsilon_3} =&C(1+\epsilon_1+\epsilon_3)\int_0^t\gamma^2(s)\frac{s^2}{\sigma^4}\mathrm{d} s
			=C(1+\epsilon_1+\epsilon_3)\left(t-\frac{2\sigma^2}{\sigma_0^2}\ln \frac{\sigma_0^2t+\sigma^2}{\sigma^2}-\frac{\sigma^4}{(\sigma_0^2t+\sigma^2)\sigma_0^2}+\frac{\sigma^2}{\sigma_0^2}\right).
		\end{aligned}
	\end{equation*}
 
 Because $\mu\sim N\left(y_0,\sigma_0^2\right)$ under probability measure $\mathbb{P}$, by \citet[Lemma C.2]{guan2024equil}, Condition \eqref{condition2} is sufficient to ensure $\mathbb{E}^{\mathbb{P}}\left[\mathrm{e}^{C_{t,C,\epsilon_1, \epsilon_3} \mu ^2}\right]<\infty$. Besides, by Lemma \ref{lemma:w}, Condition \eqref{condition1} is sufficient to ensure $\mathbb{E}^{\mathbb{P}}\left[\mathrm{exp}\left\{C(1+\epsilon_2+\frac{1}{\epsilon_3})\int_0^t\gamma^2(s) \frac{ W^2(s)}{\sigma^2}\mathrm{d} s\right\}\right]<\infty$. Combining \eqref{equ:yestimate} yields that the lemma holds.

 
\end{proof}

The following proposition shows that $\pi^{*}$ is admissible.

\begin{proposition}
	\label{admissible}
	Suppose that the following conditions hold:
	$\exists \delta_1, \delta_7, \delta_8>1, \frac{1}{\delta_7}+\frac{1}{\delta_8}=1, \epsilon_3 >0$, such that
	\begin{itemize}
		\item $2(2\delta_1^2\delta_7-\delta_1 )\delta_8(1+\frac{1}{\epsilon_3})T^2\frac{\sigma_0^4}{\sigma^4}<1$;
		\item $2(2\delta_1^2\delta_7-\delta_1 )\delta_8\frac{1}{\sigma^2}(1+\epsilon_3)(T\sigma_0^2-2\sigma^2\ln \frac{\sigma_0^2T+\sigma^2}{\sigma^2}-\frac{\sigma^4}{\sigma_0^2T+\sigma^2}+\sigma^2)<1$ \\ (or $2(2\delta_1^2\delta_7-\delta_1 )\delta_8(1+\epsilon_3)\frac{\sigma_0^6T^3}{3\sigma^6}<1$).
	\end{itemize}
	For any $\tilde \mu \in \mathcal{M}$ and the feasible strategy $\pi^{*}$, let $X^{\pi^{*}}$ be the unique strong solution of the following SDE:
	\begin{equation*}
		\left\{\begin{array}{l}
			\mathrm{d}X^{\pi^{*}}(t)=rX^{\pi^{*}} (t)\mathrm{d}t+\pi^{*}(t)(\tilde\mu(t)-r)\mathrm{d}t+\sigma\pi^{*}(t)\mathrm{d}W^{\tilde \mu}(t),\ t\in [0, T],\\
			X^{\pi^{*}}(0)=x_0.
		\end{array}\right.
	\end{equation*}
	Then 
	$$
	\left\{\int_0^t \varphi_x\left(s, X^{\pi^{*}}(s), Y(s)\right)\sigma \pi^{*}(s)\mathrm{d} W^{\tilde  \mu}(s)\right\}_{t \in[0, T]} \text { and }\left\{\int_0^t \varphi_y\left(s, X^{\pi^{*}}(s), Y(s)\right) \frac{\gamma(s)}{\sigma} \mathrm{d} W^{\tilde \mu}(s)\right\}_{t \in[0, T]}
	$$
	are martingales w.r.t. the filtration $\{\mathcal{F}_{t}^{S}\}_{0\leq t \leq T}$ under the  probability measure $\mathbb{Q}^{\tilde  \mu}$. Therefore, $\pi^{*}$ is admissible.
\end{proposition}
\begin{proof}
	We only need to show 
	\begin{equation}\label{inqu:phix}
		\begin{aligned}
			&\mathbb{E}^{\mathbb{Q}^{\tilde \mu}}\left[\int_0^T \left|\varphi_x\left(t, X^{\pi^{*}} (t), Y(t)\right)\sigma \pi^{*}(t)\right|^2\mathrm{d}t\right]<+\infty,\\
			&\mathbb{E}^{\mathbb{Q}^{\tilde \mu}}\left[\int_0^T \left|\varphi_y\left(t, X^{\pi^{*}} (t), Y(t)\right) \frac{\gamma(t)}{\sigma}\right|^2 \mathrm{d}t\right]<+\infty.
		\end{aligned}
	\end{equation}
	Letting $h(t)=k \mathrm{e}^{r(T-t)}$, there exist two positive constants $B_1$ and $B_2$ such that
	\begin{equation*}
		\begin{aligned}	
			\left|\varphi_x\left(t, X^{\pi^{*}}(t), Y(t)\right)\sigma \pi^{*}(t)\right|^2&=h^2(t)\left|\varphi\left(t, X^{\pi^{*}}(t), Y(t)\right)\sigma \pi^{*}(t)\right|^2\\
			&\leq B_1(1+|Y(t)|^2)\mathrm{e}^{-2h(t)X^{\pi^{*}}(t)+2f(t,Y(t))},\\
			\left|\varphi_y\left(t, X^{\pi^{*}}(t), Y(t)\right) \frac{\gamma(t)}{\sigma}\right|^2&=\left|f_y(t,Y(t))\varphi\left(t, X^{\pi^{*}}(t), Y(t)\right)\frac{\gamma(t)}{\sigma}\right|^2\\
			&\leq  B_2(1+|Y(t)|^2)\mathrm{e}^{-2h(t)X^{\pi^{*}}(t)+2f(t,Y(t))}.
		\end{aligned}
	\end{equation*}
	Therefore, we only need to show 
	$$
	\begin{aligned}
		\mathbb{E}^{\mathbb{Q}^{\tilde \mu}}\left[\int_0^T (1+|Y(t)|^2)\mathrm{e}^{-2h(t)X^{\pi^{*}}(t)+2f(t,Y(t))}\mathrm{d}t\right]&=\int_0^T\mathbb{E}^{\mathbb{Q}^{\tilde \mu}}\left[(1+|Y(t)|^2)\mathrm{e}^{-2h(t)X^{\pi^{*}}(t)+2f(t,Y(t))}\right]\mathrm{d}t\\
		&<+\infty.
	\end{aligned}
	$$
	Using Itô's formula and PDE (\ref{pde1}) that $f(t,y)$ satisfies, we obtain
	\begin{equation*}
		\begin{aligned}
			\mathrm{e}^{-rt}X^{\pi^{*}}(t)=&\ x_0+\int_0^t \mathrm{e}^{-rs} \pi^{*}(s)(\tilde\mu(s)-r)\mathrm{d} s+\int_0^t \mathrm{e}^{-rs}\sigma \pi^{*}(s)\mathrm{d} W^{\tilde \mu}(s),\\
			\mathrm{e}^{-h(t)X^{\pi^{*}}(t)}=&\ \mathrm{exp}\left\{-k\mathrm{e}^{rT}x_0-k\mathrm{e}^{rT}\int_0^t \mathrm{e}^{-rs} \pi^{*}(s)(\tilde\mu(s)-r)\mathrm{d} s-k\mathrm{e}^{rT}\int_0^t \mathrm{e}^{-rs}\sigma \pi^{*}(s)\mathrm{d} W^{\tilde \mu}(s)\right\},\\
			\mathrm{e}^{f(t,Y(t))}=&\ \mathrm{exp}\bigg\{f(0,y_0)+\int_0^t f_y(s,Y(s))\frac{\gamma(s)}{\sigma}\mathrm{d} W^{\tilde \mu}(s)+\int_0^tf_y(s,Y(s))\frac{\gamma(s)}{\sigma^2}(\tilde\mu(s)-r)\\&\ +\tilde g(s,Y(s))\mathrm{d} s\bigg\},
		\end{aligned}
	\end{equation*}
	where $\tilde g$ is defined by \eqref{equ:gt} in the proof of Theorem \ref{theorem1}.
	
	Combining with the expression \eqref{equ:pihat} of $\pi^{*}$, we have  that
	\begin{equation*}
		\begin{aligned}
			\mathrm{e}^{-h(t)X^{\pi^{*}}(t)+f(t,Y(t))}&=B_3\mathrm{exp}\left\{-k\mathrm{e}^{rT}\int_0^t \mathrm{e}^{-rs} \pi^{*}(s)(\tilde\mu(s)-r)\mathrm{d} s-k\mathrm{e}^{rT}\int_0^t \mathrm{e}^{-rs}\sigma \pi^{*}(s)\mathrm{d} W^{\tilde \mu}(s)\right\}\\
			&\quad\times \mathrm{exp}\left\{\int_0^t f_y(s,Y(s))\frac{\gamma(s)}{\sigma}\mathrm{d} W^{\tilde \mu}(s)+\int_0^tf_y(s,Y(s))\frac{\gamma(s)}{\sigma^2}(\tilde\mu(s)-r)+\tilde g(s,Y(s))\mathrm{d} s\right\}\\
			&=B_3 \mathrm{exp}\left\{\int_0^t\tilde g(s,Y(s))\mathrm{d} s+\int_0^t\tilde g_2(s,Y(s))(\tilde\mu(s)-r)\mathrm{d} s+\int_0^t\tilde g_2(s,Y(s))\sigma\mathrm{d}  W^{\tilde \mu}(s)\right\},                      
		\end{aligned}
	\end{equation*}
	where 
\begin{equation}\label{equ:g2t}
\tilde g_2(t,y)\triangleq\frac{r-y+a\sqrt{\gamma(t)}}{\sigma^2}\mathrm{I}_{\{r-y+a\sqrt{\gamma(t)}\leq 0\}}+\frac{r-y-a\sqrt{\gamma(t)}}{\sigma^2}\mathrm{I}_{\{r-y-a\sqrt{\gamma(t)}\geq 0\}}\end{equation} and $B_3=\mathrm{e}^{-k\mathrm{e}^{rT}x_0+f(0,y_0)}$ is a positive constant. 
 
 As such, we have $\tilde g_2(s,Y(s))(\tilde\mu(s)-r)\leq -2\tilde g(s,Y(s))$ and $\tilde g_2(s,Y(s))(Y(s)-r)\leq -2\tilde g(s,Y(s))$. 
 
 Then  \begin{equation*}
		\begin{aligned}
			&\mathbb{E}^{\mathbb{Q}^{\tilde  \mu}}\left[(1+|Y(t)|^2)\mathrm{e}^{-2h(t)X^{\pi^{*}}(t)+2f(t,Y(t))}\right]\\
			=&\mathbb{E}^{\mathbb{P}}\left[(1+|Y(t)|^2)\mathrm{e}^{-2h(t)X^{\pi^{*}}(t)+2f(t,Y(t))}\mathrm{exp}\left\{-\int_0^t \frac{Y(s)-\tilde \mu(s)}{\sigma} \mathrm{d} W^{S}(s)-\frac{1}{2}\int_0^t \left(\frac{Y(s)-\tilde \mu(s)}{\sigma} \right)^2\mathrm{d} s\right\}\right]\\
			=&B_3^2\mathbb{E}^{\mathbb{P}}\left[(1+|Y(t)|^2)\mathrm{exp}\left\{\int_0^t2\tilde g(s,Y(s))\mathrm{d} s+\int_0^t2\tilde g_2(s,Y(s))(\tilde\mu(s)-r)\mathrm{d} s+\int_0^t2\tilde g_2(s,Y(s))\sigma\mathrm{d}  W^{\tilde \mu}(s)\right\}\right.\\
			&\left.\times\mathrm{exp}\left\{-\int_0^t \frac{Y(s)-\tilde \mu(s)}{\sigma} \mathrm{d} W^{S}(s)-\frac{1}{2}\int_0^t \left(\frac{Y(s)-\tilde \mu(s)}{\sigma} \right)^2\mathrm{d} s\right\}\right]\\
			=&B_3^2\mathbb{E}^{\mathbb{P}}\left[(1+|Y(t)|^2)\mathrm{exp}\left\{\int_0^t2\tilde g(s,Y(s))\mathrm{d} s+\int_0^t2\tilde g_2(s,Y(s))(\tilde\mu(s)-r)\mathrm{d} s\right\}\right.\\
			&\left. \times\mathrm{exp}\left\{\int_0^t2\tilde g_2(s,Y(s))\sigma(\frac{Y(s)-\tilde \mu(s)}{\sigma}\mathrm{d} s+\mathrm{d}  W^{S}(s))\right\}\right.\\
			&\left. \times\mathrm{exp}\left\{-\int_0^t \frac{Y(s)-\tilde \mu(s)}{\sigma} \mathrm{d} W^{S}(s)-\frac{1}{2}\int_0^t \left(\frac{Y(s)-\tilde \mu(s)}{\sigma} \right)^2\mathrm{d} s\right\}\right]\\
			=&B_3^2\mathbb{E}^{\mathbb{P}}\left[(1+|Y(t)|^2)\mathrm{exp}\left\{\int_0^t2\tilde g(s,Y(s))\mathrm{d} s+\int_0^t2\tilde g_2(s,Y(s))(Y(s)-r)\mathrm{d} s+\int_0^t2\tilde g_2(s,Y(s))\sigma\mathrm{d}  W^{S}(s)\right\}\right.\\
			&\left. \times\mathrm{exp}\left\{-\int_0^t \frac{Y(s)-\tilde \mu(s)}{\sigma} \mathrm{d} W^{S}(s)-\frac{1}{2}\int_0^t \left(\frac{Y(s)-\tilde \mu(s)}{\sigma} \right)^2\mathrm{d} s\right\}\right]\\
			\leq& B_3^2\mathbb{E}^{\mathbb{P}}\left[(1+|Y(t)|^2)\mathrm{exp}\left\{\int_0^t-2\tilde g(s,Y(s))\mathrm{d} s+\int_0^t2\tilde g_2(s,Y(s))\sigma\mathrm{d}  W^{S}(s)\right\}\right.\\
			&\left. \times\mathrm{exp}\left\{-\int_0^t \frac{Y(s)-\tilde \mu(s)}{\sigma} \mathrm{d} W^{S}(s)-\frac{1}{2}\int_0^t \left(\frac{Y(s)-\tilde \mu(s)}{\sigma} \right)^2\mathrm{d} s\right\}\right].
		\end{aligned}
	\end{equation*}
Using Hölder's inequality, for any constants $\delta_1, \delta_2>1$ with $\frac{1}{\delta_1}+\frac{1}{\delta_2}=1$, we have
	\begin{equation}\label{inequ:0}
		\begin{aligned}
			&\mathbb{E}^{\mathbb{Q}^{\tilde  \mu}}\left[(1+|Y(t)|^2)\mathrm{e}^{-2h(t)X^{\pi^{*}}(t)+2f(t,Y(t))}\right]\\
			\leq& B_3^2\left\{\mathbb{E}^{\mathbb{P}}\left[\mathrm{exp}\left\{\int_0^t-2\delta_1 \tilde g(s,Y(s))\mathrm{d} s+\int_0^t2 \delta_1 \tilde g_2(s,Y(s))\sigma\mathrm{d}  W^{S}(s)\right\}\right]\right\}^{\frac{1}{\delta_1}}\\
			&\times\left\{\mathbb{E}^{\mathbb{P}}\left[(1+|Y(t)|^2)^{\delta_2}\mathrm{exp}\left\{-\int_0^t \delta_2\frac{Y(s)-\tilde \mu(s)}{\sigma} \mathrm{d} W^{S}(s)-\frac{1}{2}\int_0^t\delta_2 \left(\frac{Y(s)-\tilde \mu(s)}{\sigma} \right)^2\mathrm{d} s\right\}\right]\right\}^{\frac{1}{\delta_2}}.
		\end{aligned}
	\end{equation}
Based on the above discussion, to prove \eqref{inqu:phix}, it suffices to show that both terms on the right-hand side of \eqref{inequ:0} are finite. First, we demonstrate that the second term on the right-hand side of \eqref{inequ:0} is finite.  Using Hölder's inequality, for any constants  $\delta_3, \delta_4>1$ with  $\frac{1}{\delta_3}+\frac{1}{\delta_4}=1$, 
	\begin{equation}\label{inequ:1}
		\begin{aligned}	&\mathbb{E}^{\mathbb{P}}\left[(1+|Y(t)|^2)^{\delta_2}\mathrm{exp}\left\{-\int_0^t \delta_2\frac{Y(s)-\tilde \mu(s)}{\sigma} \mathrm{d} W^{S}(s)-\frac{1}{2}\int_0^t\delta_2 \left(\frac{Y(s)-\tilde \mu(s)}{\sigma} \right)^2\mathrm{d} s\right\}\right]\\
			\leq& \left\{\mathbb{E}^{\mathbb{P}}\left[(1+|Y(t)|^2)^{\delta_2\delta_3}\right]\right\}^{\frac{1}{\delta_3}}\times\left\{\mathbb{E}^{\mathbb{P}}\left[\mathrm{exp}\left\{-\int_0^t \delta_2\delta_4\frac{Y(s)-\tilde \mu(s)}{\sigma} \mathrm{d} W^{S}(s)-\frac{1}{2}\int_0^t\delta_2 \delta_4\left(\frac{Y(s)-\tilde \mu(s)}{\sigma} \right)^2\mathrm{d} s\right\}\right]\right\}^{\frac{1}{\delta_4}}.
		\end{aligned}
	\end{equation}
Because $Y(t)$ is normally distributed under the probability measure $\mathbb{P}$, there exists a constant $B_4>0$ such that \begin{equation}\label{inequ:2}
\mathbb{E}^{\mathbb{P}}\left[(1+|Y(t)|^2)^{\delta_2\delta_3}\right]\leq B_4<+\infty.
\end{equation}

       Using Hölder's inequality, for any $\delta_5, \delta_6>1$ with  $\frac{1}{\delta_5}+\frac{1}{\delta_6}=1$, we have 
	\begin{equation}
	\begin{aligned}	&\mathbb{E}^{\mathbb{P}}\left[\mathrm{exp}\left\{-\int_0^t \delta_2\delta_4\frac{Y(s)-\tilde \mu(s)}{\sigma} \mathrm{d} W^{S}(s)-\frac{1}{2}\int_0^t\delta_2 \delta_4\left(\frac{Y(s)-\tilde \mu(s)}{\sigma} \right)^2\mathrm{d} s\right\}\right]\\
		\leq &\left\{\mathbb{E}^{\mathbb{P}}\left[\mathrm{exp}\left\{-\int_0^t \delta_2\delta_4\delta_5\frac{Y(s)-\tilde \mu(s)}{\sigma} \mathrm{d} W^{S}(s)-\frac{1}{2}\int_0^t\delta_2^2 \delta_4^2\delta_5^2\left(\frac{Y(s)-\tilde \mu(s)}{\sigma} \right)^2\mathrm{d} s\right\}\right]\right\}^{\frac{1}{\delta_5}} \\&\times\left\{\mathbb{E}^{\mathbb{P}}\left[\mathrm{exp}\left\{\frac{1}{2}\int_0^t(\delta_2^2 \delta_4^2\delta_5-\delta_2 \delta_4)\delta_6\left(\frac{Y(s)-\tilde \mu(s)}{\sigma} \right)^2\mathrm{d} s\right\}\right]\right\}^{\frac{1}{\delta_6}}.
	\end{aligned}\label{inequ:3}
	\end{equation}
	Because $$\left\{\mathrm{exp}\left\{-\int_0^t \delta_2\delta_4\delta_5\frac{Y(s)-\tilde \mu(s)}{\sigma} \mathrm{d} W^{S}(s)-\frac{1}{2}\int_0^t\delta_2^2 \delta_4^2\delta_5^2\left(\frac{Y(s)-\tilde \mu(s)}{\sigma} \right)^2\mathrm{d} s\right\}:0\le t\le T\right\}$$ is a supermartingale w.r.t. the filtration $\{\mathcal{F}_{t}^{S}\}_{0\leq t \leq T}$ under the  probability measure $\mathbb{P}$, we have 
 \begin{equation}\label{inequ:4}
\mathbb{E}^{\mathbb{P}}\left[\mathrm{exp}\left\{-\int_0^t \delta_2\delta_4\delta_5\frac{Y(s)-\tilde \mu(s)}{\sigma} \mathrm{d} W^{S}(s)-\frac{1}{2}\int_0^t\delta_2^2 \delta_4^2\delta_5^2\left(\frac{Y(s)-\tilde \mu(s)}{\sigma} \right)^2\mathrm{d} s\right\}\right]\leq 1.
\end{equation}
	Meanwhile, due to the uniform boundedness of $\Lambda_{t,y}$, there exists a constant $B_5>0$ such that 
 \begin{equation}\label{inequ:5}
\mathbb{E}^{\mathbb{P}}\left[\mathrm{exp}\left\{\frac{1}{2}\int_0^t(\delta_2^2 \delta_4^2\delta_5-\delta_2 \delta_4)\delta_6\left(\frac{Y(s)-\tilde \mu(s)}{\sigma} \right)^2\mathrm{d} s\right\}\right]\leq B_5<\infty.
\end{equation}
 
  Combining inequalities \eqref{inequ:1}, \eqref{inequ:2}, \eqref{inequ:3}, \eqref{inequ:4}, and \eqref{inequ:5}, there exists  a constant $B_6>0$ such that
$$\mathbb{E}^{\mathbb{P}}\left[(1+|Y(t)|^2)^{\delta_2}\mathrm{exp}\left\{-\int_0^t \delta_2\frac{Y(s)-\tilde \mu(s)}{\sigma} \mathrm{d} W^{S}(s)-\frac{1}{2}\int_0^t\delta_2 \left(\frac{Y(s)-\tilde \mu(s)}{\sigma} \right)^2\mathrm{d} s\right\}\right]\leq B_6 < +\infty$$
and the second term on the right side of \eqref{inequ:0} is finite. Then it remains to prove that the first term on the right side of \eqref{inequ:0} is finite, i.e., estimate the expectation $$\mathbb{E}^{\mathbb{P}}\left[\mathrm{exp}\left\{\int_0^t-2\delta_1 \tilde g(s,Y(s))\mathrm{d} s+\int_0^t2 \delta_1 \tilde g_2(s,Y(s))\sigma\mathrm{d}  W^{S}(s)\right\}\right].$$
	
 Using Hölder's inequality, for any   constants $\delta_7, \delta_8>1$ with $\frac{1}{\delta_7}+\frac{1}{\delta_8}=1$, we have 
	\begin{equation}
		\begin{aligned}
  &\mathbb{E}^{\mathbb{P}}\left[\mathrm{exp}\left\{\int_0^t-2\delta_1 \tilde g(s,Y(s))\mathrm{d} s+\int_0^t2 \delta_1 \tilde g_2(s,Y(s))\sigma\mathrm{d}  W^{S}(s)\right\}\right]\\
			\leq& 
   \left\{\mathbb{E}^{\mathbb{P}}\left[\mathrm{exp}\left\{\int_0^t-4\delta_1^2\delta_7^2 \tilde g(s,Y(s))\mathrm{d} s+\int_0^t2 \delta_1 \delta_7\tilde g_2(s,Y(s))\sigma\mathrm{d}  W^{S}(s)\right\}\right]\right\}^{\frac{1}{\delta_7}} \\&\times\left\{\mathbb{E}^{\mathbb{P}}\left[\mathrm{exp}\left\{\int_0^t(4\delta_1^2\delta_7-2\delta_1 )\delta_8\tilde g(s,Y(s))\mathrm{d} s\right\}\right]\right\}^{\frac{1}{\delta_8}}.
		\end{aligned}\label{inqu:21}
	\end{equation}
	 
   Because $$\left\{\mathrm{exp}\left\{\int_0^t-4\delta_1^2\delta_7^2 \tilde g(s,Y(s))\mathrm{d} s+\int_0^t2 \delta_1 \delta_7\tilde g_2(s,Y(s))\sigma\mathrm{d}  W^{S}(s)\right\}:0\le t\le T\right\}$$ is a supermartingale w.r.t. the filtration $\{\mathcal{F}_{t}^{S}\}_{0\leq t \leq T}$ under the probability measure $\mathbb{P}$, we have 
	$$\mathbb{E}^{\mathbb{P}}\left[\mathrm{exp}\left\{\int_0^t-4\delta_1^2\delta_7^2 \tilde g(s,Y(s))\mathrm{d} s+\int_0^t2 \delta_1 \delta_7\tilde g_2(s,Y(s))\sigma\mathrm{d}  W^{S}(s)\right\}\right]\leq 1.$$
	Then the first term on the right side of \eqref{inqu:21} is finite. It remains to show that the second term on the right side of \eqref{inqu:21} is finite. By the form of $\tilde g$, we  have that for any $\epsilon>0$,
 $$\mathbb{E}^{\mathbb{P}}\left[\mathrm{exp}\left\{\int_0^t(4\delta_1^2\delta_7-2\delta_1 )\delta_8\tilde g(s,Y(s))\mathrm{d} s\right\}\right]\leq \mathbb{E}^{\mathbb{P}}\left[\mathrm{exp}\left\{\int_0^t(4\delta_1^2\delta_7-2\delta_1 )\delta_8[(\frac{1}{2\sigma^2}+\epsilon)Y^2(s)+C_{\epsilon}]\mathrm{d} s\right\}\right],$$
	where  $C_{\epsilon}$ is a constant depend on $\epsilon$.

Therefore, we only need to show that $\mathbb{E}^{\mathbb{P}}\left[\mathrm{exp}\left\{C\int_0^tY^2(s)\mathrm{d} s\right\}\right]$ is finite, where $C=(4\delta_1^2\delta_7-2\delta_1 )\delta_8(\frac{1}{2\sigma^2}+\epsilon)>0$ is a constant. 

Based on Lemma \ref{lemma:y2}, if the following conditions hold:
	$\exists \delta_1, \delta_7, \delta_8>1, \frac{1}{\delta_7}+\frac{1}{\delta_8}=1, \epsilon, \epsilon_1, \epsilon_2,\epsilon_3 >0$, such that
	\begin{itemize}
		\item $2(4\delta_1^2\delta_7-2\delta_1 )\delta_8(\frac{1}{2\sigma^2}+\epsilon)(1+\epsilon_2+\frac{1}{\epsilon_3})T^2\frac{\sigma_0^4}{\sigma^2}<1$;
		\item $2(4\delta_1^2\delta_7-2\delta_1 )\delta_8(\frac{1}{2\sigma^2}+\epsilon)(1+\epsilon_1+\epsilon_3)(T\sigma_0^2-2\sigma^2\ln \frac{\sigma_0^2T+\sigma^2}{\sigma^2}-\frac{\sigma^4}{\sigma_0^2T+\sigma^2}+\sigma^2)<1$ \\ (or $2(4\delta_1^2\delta_7-2\delta_1 )\delta_8(\frac{1}{2\sigma^2}+\epsilon)(1+\epsilon_1+\epsilon_3)\frac{\sigma_0^6T^3}{3\sigma^4}<1$),
	\end{itemize}
	then 
	$$\mathbb{E}^{\mathbb{Q}^{\tilde \mu}}\left[\int_0^T (1+|Y(t)|^2)\mathrm{e}^{-2h(t)X^{\pi^{*}}(t)+2f(t,Y(t))}\mathrm{d}t\right]<\infty.$$
	These conditions is equivalent to 
	$\exists \delta_1, \delta_7, \delta_8>1, \frac{1}{\delta_7}+\frac{1}{\delta_8}=1, \epsilon_3 >0$ such that
	\begin{itemize}
		\item $2(2\delta_1^2\delta_7-\delta_1 )\delta_8(1+\frac{1}{\epsilon_3})T^2\frac{\sigma_0^4}{\sigma^4}<1$;
		\item $2(2\delta_1^2\delta_7-\delta_1 )\delta_8\frac{1}{\sigma^2}(1+\epsilon_3)(T\sigma_0^2-2\sigma^2\ln \frac{\sigma_0^2T+\sigma^2}{\sigma^2}-\frac{\sigma^4}{\sigma_0^2T+\sigma^2}+\sigma^2)<1$ \\ (or $2(2\delta_1^2\delta_7-\delta_1 )\delta_8(1+\epsilon_3)\frac{\sigma_0^6T^3}{3\sigma^6}<1$).
	\end{itemize}
	
	Thus, Proposition \ref{admissible} is proved.
\end{proof}

\begin{remark}
Under the assumptions of Proposition \ref{admissible}, using a similar approach, it can be shown that there exists a sufficiently small $\delta > 0$ such that any progressively measurable (relative to $\{\mathcal{F}_{t}^{S}\}_{0\leq t \leq T}$) investment strategy process $\pi=\left\{\pi(t): 0 \leq t \leq T\right\}$ satisfying $|\pi(t)-\pi^{*}(t)| \leq \delta (1 + |Y(t)|)$ for all $t \in [0, T]$ is admissible. This result highlights the abundance of admissible strategies.
\end{remark}
\begin{remark}

The assumptions in Proposition \ref{admissible} are quite mild. If the time span $T$ is relatively short, these assumptions naturally hold. Due to the investor's learning effect, she gradually becomes more sophisticated, making the time required to adopt a robust investment strategy relatively short.  Consequently, the assumptions are reasonable.
\end{remark}

\subsection{Verification theorem}
In this subsection, we give the verification theorem and prove the optimality of the candidate robust optimal investment strategy $\pi^{*}$.

Let 
\begin{equation*}
	\begin{aligned}
		\hat\mu(t , y;\pi) = \begin{cases} y-a\sqrt{\gamma(t)} ,  \qquad &\pi>  \frac{f_y(t,y)\gamma(t)}{k \mathrm{e}^{r(T-t)}\sigma^2}, \\
			y ,\qquad  &\pi=  \frac{f_y(t,y)\gamma(t)}{k \mathrm{e}^{r(T-t)}\sigma^2},\\
			y+a\sqrt{\gamma(t)},\qquad &\pi<  \frac{f_y(t,y)\gamma(t)}{k \mathrm{e}^{r(T-t)}\sigma^2}. \end{cases} \\
	\end{aligned}
\end{equation*}
Note that we can define $\hat\mu(t , y;\pi) \in \Lambda_{t,y}$ arbitrarily if $\pi=  \frac{f_y(t,y)\gamma(t)}{k \mathrm{e}^{r(T-t)}\sigma^2}$.

Denote $\tilde\mu^{*}=\{\tilde\mu^{*}(t)=\hat\mu(t , Y(t);\pi^{*}(t)):0\leq t\leq T \}$, we know that $\tilde\mu^{*}\in \mathcal{M}$.
Let $X^{\pi^{*}}$ be the unique strong solution of the following SDE:
\begin{equation*}
	\left\{\begin{array}{l}
		\mathrm{d}X^{\pi^{*}}(t)=rX^{\pi^{*}}(t)\mathrm{d}t+\pi^{*}(t)(\tilde\mu^{*}(t)-r)\mathrm{d}t+\sigma\pi^{*}(t)\mathrm{d}W^{\tilde \mu^{*}}(t),\ t\in [0, T],\\
		X^{\pi^{*}}(0)=x_0.
	\end{array}\right.
\end{equation*}
We are now able to show the main theorem in this section.
\begin{theorem}[Verification Theorem]
	Under the assumptions of Proposition \ref{admissible}, for the robust optimal control problem \eqref{obj1},  
 {$\pi^{*}$} is the robust optimal investment strategy, $\tilde{\mu}^{*}$ represents the worst-case scenario for $\mu$, and
	$$V(x_0)=\max_{\pi\in \Pi}\min_{\mathbb{Q}^{\tilde \mu}\in \mathcal{Q}}\mathbb{E}^{\mathbb{Q}^{\tilde \mu}}[U(X^{\pi}(T))]=\min_{\mathbb{Q}^{\tilde \mu}\in\mathcal{Q}}\mathbb{E}^{\mathbb{Q}^{\tilde \mu}}[U(X^{\pi^{*}}(T))]=\mathbb{E}^{\mathbb{Q}^{\tilde \mu^{*}}}[U(X^{\pi^{*}}(T))]=\varphi\left(0,x_0,y_0\right).$$
	\label{theorem2}
\end{theorem}
\begin{proof}
	From the discussion above, we know that $\varphi(t,x,y)$ is a solution of the HJBI equation (\ref{hjbi0}). Additionally, we know that for any $\pi \in \mathbb{R}$, $\hat\mu(t , y;\pi) $ satisfies
	\begin{equation*}
		\begin{aligned}
			\inf_{\mu \in \Lambda_{t,y}}\{\mu (\varphi_x\pi+\varphi_y\frac{\gamma(t)}{\sigma^2})\}=\hat\mu(t , y;\pi) (\varphi_x\pi+\varphi_y\frac{\gamma(t)}{\sigma^2}).
		\end{aligned}
	\end{equation*}
	For any admissible strategy $\pi$, define  $\check\mu=\{\check\mu(t)=\hat\mu(t,Y(t);\pi(t)):0\leq t \leq T\}\in \mathcal{M}$, let $X^{\pi}$ be the unique strong solution under $\mathbb{Q}^{\check \mu}$ of the following SDE:
	\begin{equation*}
		\left\{\begin{array}{l}
			\mathrm{d} X^{\pi}(t)=r X^{\pi}(t)\mathrm{d}t+\pi(t)(\check\mu(t)-r)\mathrm{d}t+\sigma\pi(t)\mathrm{d}W^{\check \mu}(t),\ t\in [0, T],\\
			X^{\pi}(0)=x_0.
		\end{array}\right.
	\end{equation*}	
	Using Itô's lemma, 
	\begin{equation*}
		\begin{aligned}
			\mathrm{d}\varphi\left(t, X^{\pi}(t),Y(t)\right)&=\varphi_t\left(t, X^{\pi}(t),Y(t)\right)\mathrm{d}t\\
			&\quad +\varphi_x\left(t, X^{\pi}(t),Y(t)\right)\left[r X^{\pi}(t)\mathrm{d}t+\pi(t)(\check\mu(t)-r)\mathrm{d}t+\sigma\pi(t)\mathrm{d}W^{\check \mu}(t)\right]\\
			&\quad+\varphi_y\left(t, X^{\pi}(t),Y(t)\right)\left\{\gamma(t) \sigma^{-2}\left[(\check\mu(t)-Y(t))\mathrm{d} t+\sigma \mathrm{d} W^{\check\mu}(t)\right]\right\}\\
			&\quad+\frac{1}{2}\varphi_{xx}\left(t, X^{\pi}(t),Y(t)\right)\sigma^2\pi^2(t)\mathrm{d}t+\frac{1}{2}\varphi_{yy}\left(t,X^{\pi}(t),Y(t)\right)\frac{\gamma^2(t)}{\sigma^2}\mathrm{d}t\\
			&\quad+\varphi_{xy}\left(t,X^{\pi}(t),Y(t)\right)\gamma(t)\pi(t)\mathrm{d}t.
		\end{aligned}
	\end{equation*}
	As $\varphi(t,x,y)$ satisfies (\ref{hjbi0}), we have 
	\begin{equation*}
		\begin{aligned}
			\mathrm{d}\varphi\left(t, X^{\pi}(t),Y(t)\right)\leq \varphi_x\left(t, X^{\pi}(t),Y(t)\right)\sigma\pi(t)\mathrm{d}W^{\check \mu}(t)+\varphi_y\left(t, X^{\pi}(t),Y(t)\right)\frac{\gamma(t)} {\sigma} \mathrm{d} W^{\check\mu}(t).
		\end{aligned}
	\end{equation*}
	Based on Definition \ref{def} (ii) of the admissible investment strategy, $$\mathbb{E}^{\mathbb{Q}^{\check \mu}}[U(X^{\pi}(T))]=\mathbb{E}^{\mathbb{Q}^{\check \mu}}\left[\varphi\left(T, X^{\pi}(T),Y(T)\right)\right]\leq \varphi\left(0,x_0,y_0\right).$$ Thus 
	$$V(x_0)=\max_{\pi\in\Pi}\min_{\mathbb{Q}^{\tilde \mu}\in\mathcal{Q}}\mathbb{E}^{\mathbb{Q}^{\tilde \mu}}[U(X^{\pi}(T))]\leq \varphi\left(0,x_0,y_0\right).$$
	On the other hand, denote
	\begin{equation*}
		\begin{aligned}
			\psi(t,\mu)&=\varphi_t\left(t,X^{\pi^{*}}(t),Y(t)\right)+\varphi_x\left(t,X^{\pi^{*}}(t),Y(t)\right)\left[rX^{\pi^{*}}(t)+\pi^{*}(t)(\mu-r)\right]\\
			&\quad+\varphi_y\left(t,X^{\pi^{*}}(t),Y(t)\right)\left[\gamma(t) \sigma^{-2}(\mu-Y(t))\right]\\
			&\quad+\frac{1}{2}\varphi_{xx}\left(t,X^{\pi^{*}}(t),Y(t)\right)\sigma^2(\pi^{*})^2(t)+\frac{1}{2}\varphi_{yy}\left(t,X^{\pi^{*}}(t),Y(t)\right)\frac{\gamma^2(t)}{\sigma^2}\\
			&\quad+\varphi_{xy}\left(t,X^{\pi^{*}}(t),Y(t)\right)\gamma(t)\pi^{*}(t).
		\end{aligned}
	\end{equation*}
	Then 
	\begin{equation*}
		\begin{aligned}
			V(x_0)&=\max_{\pi\in\Pi}\min_{\mathbb{Q}^{\tilde \mu}\in\mathcal{Q}}\mathbb{E}^{\mathbb{Q}^{\tilde \mu}}[U(X^{\pi}(T))]
			\geq \min_{\mathbb{Q}^{\tilde \mu}\in\mathcal{Q}}\mathbb{E}^{\mathbb{Q}^{\tilde \mu}}[U(X^{\pi^{*}}(T))]
			=\min_{\mathbb{Q}^{\tilde \mu}\in\mathcal{Q}}\mathbb{E}^{\mathbb{Q}^{\tilde \mu}}\left[\varphi\left(T,X^{\pi^{*}}(T),Y(T)\right)\right]\\
			&=\min_{\mathbb{Q}^{\tilde \mu}\in\mathcal{Q}}\mathbb{E}^{\mathbb{Q}^{\tilde \mu}}\left[\int_0^T \psi(t,\tilde\mu(t))\mathrm{d}t+\int_0^T\varphi_x\left(t,X^{\pi^{*}} (t),Y(t)\right)\sigma\pi^{*}(t)\mathrm{d} W^{\tilde\mu}(t)\right.\\&\left.\quad+\int_0^T\varphi_y\left(t,X^{\pi^{*}} (t),Y(t)\right)\frac{\gamma(t)} {\sigma} \mathrm{d} W^{\tilde\mu}(t)\right] +\varphi\left(0,x_0,y_0\right)\\
			&=\min_{\mathbb{Q}^{\tilde \mu}\in\mathcal{Q}}\mathbb{E}^{\mathbb{Q}^{\tilde \mu}}\left[\int_0^T \psi(t,\tilde\mu(t))\mathrm{d}t\right]+\varphi\left(0,x_0,y_0\right)\\
			&\geq \min_{\mathbb{Q}^{\tilde \mu}\in\mathcal{Q}}\mathbb{E}^{\mathbb{Q}^{\tilde \mu}}\left[\int_0^T \min_{\mu(t)\in \Lambda_{t,Y(t)}}\psi(t,\mu(t) )\mathrm{d}t\right]+\varphi\left(0,x_0,y_0\right).
		\end{aligned}
	\end{equation*}	
	The last equality holds  due to Proposition \ref{admissible}.
	
	Based on  the definitions of $\pi^{*}, \tilde\mu^{*}$, and the fact that $\varphi(t,x,y)$ satisfies equation (\ref{hjbi0}), we have $$\min_{\mu(t)\in \Lambda_{t,Y(t)}}\psi(t,\mu(t))= \psi(t,\tilde\mu^{*}(t))=0.$$ Thus 
	\begin{equation*}
		\begin{aligned}
			V(x_0)&\geq \min_{\mathbb{Q}^{\tilde \mu}\in\mathcal{Q}}\mathbb{E}^{\mathbb{Q}^{\tilde \mu}}\left[\int_0^T \min_{\mu(t)\in \Lambda_{t,Y(t)}}\psi(t,\mu(t) )\mathrm{d}t\right]+\varphi\left(0,x_0,y_0\right)\\
			&=\min_{\mathbb{Q}^{\tilde \mu}\in\mathcal{Q}}\mathbb{E}^{\mathbb{Q}^{\tilde \mu}}\left[\int_0^T \psi(t,\tilde\mu^{*}(t) )\mathrm{d}t\right]+\varphi\left(0,x_0,y_0\right)\\
			&=\varphi\left(0,x_0,y_0\right).
		\end{aligned}
	\end{equation*}
Then it follows that
	$$V(x_0)=\max_{\pi\in\Pi}\min_{\mathbb{Q}^{\tilde \mu}\in\mathcal{Q}}\mathbb{E}^{\mathbb{Q}^{\tilde \mu}}[U(X^{\pi}(T))]=\varphi\left(0,x_0,y_0\right)=\min_{\mathbb{Q}^{\tilde \mu}\in\mathcal{Q}}\mathbb{E}^{\mathbb{Q}^{\tilde \mu}}[U(X^{\pi^{*}}(T))],$$ 
	\begin{equation*}
		\begin{aligned}
			V(x_0)&=\mathbb{E}^{\mathbb{Q}^{\tilde\mu^{*}}}\left[\int_0^T \psi(t,\tilde\mu^{*}(t) )\mathrm{d}t\right]+\varphi\left(0,x_0,y_0\right)\\
			&=\mathbb{E}^{\mathbb{Q}^{\tilde\mu^{*}}}\left[\int_0^T \psi(t,\tilde\mu^{*}(t))\mathrm{d}t+\int_0^T\varphi_x\left(t,X^{\pi^{*}} (t),Y(t)\right)\sigma\pi^{*}(t)\mathrm{d} W^{\tilde\mu^{*}}(t)\right.\\&\left.\quad+\int_0^T\varphi_y\left(t,X^{\pi^{*}} (t),Y(t)\right)\frac{\gamma(t)} {\sigma}\mathrm{d} W^{\tilde\mu^{*}}(t)\right] +\varphi\left(0,x_0,y_0\right)\\
			&=\mathbb{E}^{\mathbb{Q}^{\tilde\mu^{*}}}\left[\varphi\left(T,X^{\pi^{*}}(T),Y(T)\right)\right]\\
			&=\mathbb{E}^{\mathbb{Q}^{\tilde\mu^{*}}}[U(X^{\pi^{*}}(T))],
		\end{aligned}
	\end{equation*}	
	where the last but two equality holds based on Proposition \ref{admissible} and the last but one equality holds based on Itô's lemma. As such, we know that $\pi^{*}$ is an optimal investment strategy, $\varphi$ is the value function, and thus the theorem is proved.
\end{proof}

\subsection{Case without uncertainty}
	When $\sigma_0^2=0$, $\mu$ is a constant and $\mu\equiv y_0$, and the optimal investment strategy $\pi^{*}$ degenerates to $\bar\pi^{*}=\left\{\bar\pi^{*} (t):  0 \leq t \leq T\right\}$, where
	\begin{equation*}
		\begin{aligned}
			\bar\pi^{*} (t)=\frac{\mu-r}{k \mathrm{e}^{r(T-t)}\sigma^2}=\frac{y_0-r}{k \mathrm{e}^{r(T-t)}\sigma^2}.
		\end{aligned}
	\end{equation*}
	The robust optimal investment problem  (\ref{obj0}) degenerates to the classical optimal investment problem: 
	\begin{equation}
		\begin{aligned}
			\max_{\pi}\mathbb{E}^{\mathbb{P}}[U(X^{\pi}(T))].
		\end{aligned}
		\label{obj2}
	\end{equation}
	It is well known that $\bar\pi^{*}$ is indeed the optimal investment strategy of the optimization problem (\ref{obj2}).

\subsection{Case when $a=0$}
	When $a=0$, there is only one $\tilde\mu=\left\{\tilde \mu(t): \tilde \mu(t) =Y(t) , 0 \leq t \leq T\right\}\in \mathcal{M}$, and the robust optimal investment problem  (\ref{obj0}) is actually the optimal  investment problem (\ref{obj2}) with partial information, where the drift $\mu$ is an unknown constant whose distribution is Gaussian. Furthermore, solving \eqref{pde2} with $a=0$, we can obtain the explicit expressions of $f(t,y)$ and the optimal investment strategy $\pi_0^{*}$. To be specific,
	$$f(t,y)=f_1(t)y^2+f_2(t)y+f_3(t),$$
	where $f_1(t)$, $ f_2(t)$ and $f_3(t)$ satisfy the following ODEs: 
	\begin{equation}\label{equ:fa=0}
		\left\{\begin{array}{l}\begin{aligned}
				&f_1^{\prime}(t)  =\frac{2\gamma(t)}{\sigma^2} f_1(t)+\frac{1}{2\sigma^2},\quad \quad &f_1(T)=0 ,\\
				&f_2^{\prime}(t)  =\frac{\gamma(t)}{\sigma^2} f_2(t)-\frac{r+2rf_1(t)\gamma(t)}{\sigma^2},\quad &f_2(T)=0,\\
				&f_3^{\prime}(t)  = \frac{r^2}{2\sigma^2}-\frac{r f_2(t)\gamma(t)+f_1(t)\gamma^2(t)}{\sigma^2},\quad  &f_3(T)=0 .
		\end{aligned}\end{array}\right.
	\end{equation}
Solving the three ODEs above, we  obtain the expressions of $f_1(t)$, $ f_2(t)$ and $f_3(t)$ as follows: 
	\begin{equation*}
		\left\{\begin{array}{l}
			f_1(t)  =-\int_{t}^{T} \frac{1}{2\sigma^2}\mathrm{e}^{-\int_{t}^{s}\frac{2\gamma(u)}{\sigma^2} \mathrm{d} u} \mathrm{d} s=\frac{1}{2\gamma^2(t)}\left[ \gamma(T)-\gamma(t) \right], \\
			f_2(t)  =\int_{t}^{T} \frac{r+2rf_1(s)\gamma(s)}{\sigma^2}\mathrm{e}^{-\int_{t}^{s}\frac{\gamma(u)}{\sigma^2} \mathrm{d} u} \mathrm{d} s=\frac{1}{\sigma^2\gamma(t)}r\gamma(T)(T-t)=-2rf_1(t),\\
			f_3(t)  = \int_{t}^{T}\left[\frac{r f_2(s)\gamma(s)+f_1(s)\gamma^2(s)}{\sigma^2}-\frac{r^2}{2\sigma^2}\right]\mathrm{d} s =\frac{r^2}{\sigma^4}\gamma(T)\frac{(T-t)^2}{2}-\frac{r^2-\gamma(T)}{2\sigma^2}(T-t)-\frac{1}{2}\ln{\frac{\gamma(t)}{\gamma(T)}}.
		\end{array}\right.
	\end{equation*}
	Besides, the explicit expression of the optimal investment strategy $\pi_0^{*}=\left\{\pi_0^{*} (t): 0 \leq t \leq T\right\}$ is 
	\begin{equation}\label{equ:pi0}
		\begin{aligned}
			\pi_0^{*} (t)= \frac{f_y(t,Y(t))\gamma(t)+Y(t)-r}{k \mathrm{e}^{r(T-t)}\sigma^2}=\frac{(2f_1(t)Y(t)+f_2(t))\gamma(t)+Y(t)-r}{k \mathrm{e}^{r(T-t)}\sigma^2}=\frac{\gamma(T)}{\gamma(t)}\frac{Y(t)-r}{k \mathrm{e}^{r(T-t)}\sigma^2},
		\end{aligned}
	\end{equation}
	which is identical to the non-robust optimal investment strategy with partial information in \citet{bismuth2019portfolio}.

Additionally,  comparing \eqref{equ:pi0} and \eqref{equ:pihat}, in light of  \eqref{equ:fy1}-\eqref{equ:fy2}, we derive the following relations:
\begin{equation*}\begin{cases}
    \pi^*(t)>\pi_0^*(t)-\frac{\mathrm{e}^{-r(T-t)}}{k \sigma^2}a\sqrt{\gamma(t)},\ \text{when } Y(t)>r,\\
    \pi^*(t)<\pi_0^*(t)+\frac{\mathrm{e}^{-r(T-t)}}{k \sigma^2}a\sqrt{\gamma(t)},\ \text{when } Y(t)<r,
\end{cases}\end{equation*}
which indicate that relative to the ambiguity-neutral case, ambiguity aversion modifies the robust optimal strategy, with the adjustment not exceeding $\frac{\mathrm{e}^{-r(T-t)}}{k \sigma^2} a \sqrt{\gamma(t)}$.

\section{Numerical analysis}
\label{Numerical analysis}
In this section, we conduct a series of numerical analyses to explore the impacts of ambiguity aversion and learning on the optimal investment strategy.  We estimate the parameters of the risky asset using data of the S\&P 500 index daily closing prices from Center for Research in Security Prices (CRSP) from January 2019 to December 2023 by maximum likelihood estimation. The estimated volatility is $\sigma = 0.213$, and the expectation and variance of the Gaussian prior for $\mu$ are $y_0 = 0.174$ and $\sigma_0^2 = 0.00908$, respectively. Over the same period, the average overnight bank funding rate was $r = 0.018$. The time horizon is set to half a year, i.e., $T = 0.5$. Additionally, we set $k = 1$. Following the setup in \citet{pei2018life}, we assume that the parameter in the confidence set $\Lambda_{t,y}$ is $a = 1.96$, indicating that $\Lambda_{t,y}$ represents a confidence set with a 95\% confidence level. 
\begin{remark}
    It is straightforward to verify that the assumptions in Proposition \ref{admissible} hold if we choose $\delta_1 = 1.1$, $\delta_7 = \delta_8 = 2$, and $\epsilon_3 = 1$. 
\end{remark}


\subsection{Robust optimal feedback function $\hat\pi(t,y)$.}
\begin{figure}
	\centering
	{\includegraphics[width=0.8\linewidth]{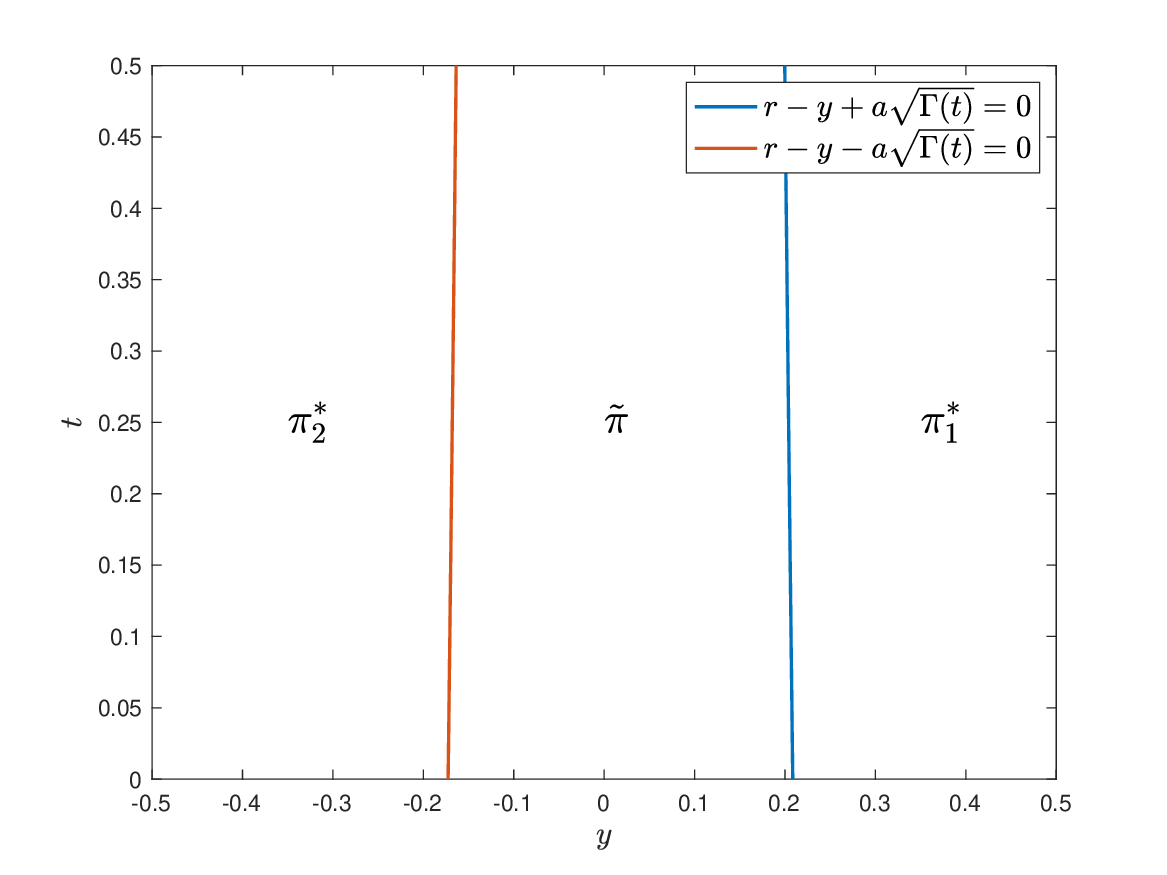}}
	\caption{Distribution of the  robust optimal feedback function $\hat\pi(t,y)$ in the $(y,t)$ plane.} \label{f1}
\end{figure}

\begin{figure}
	\centering
	{\includegraphics[width=0.8\linewidth]{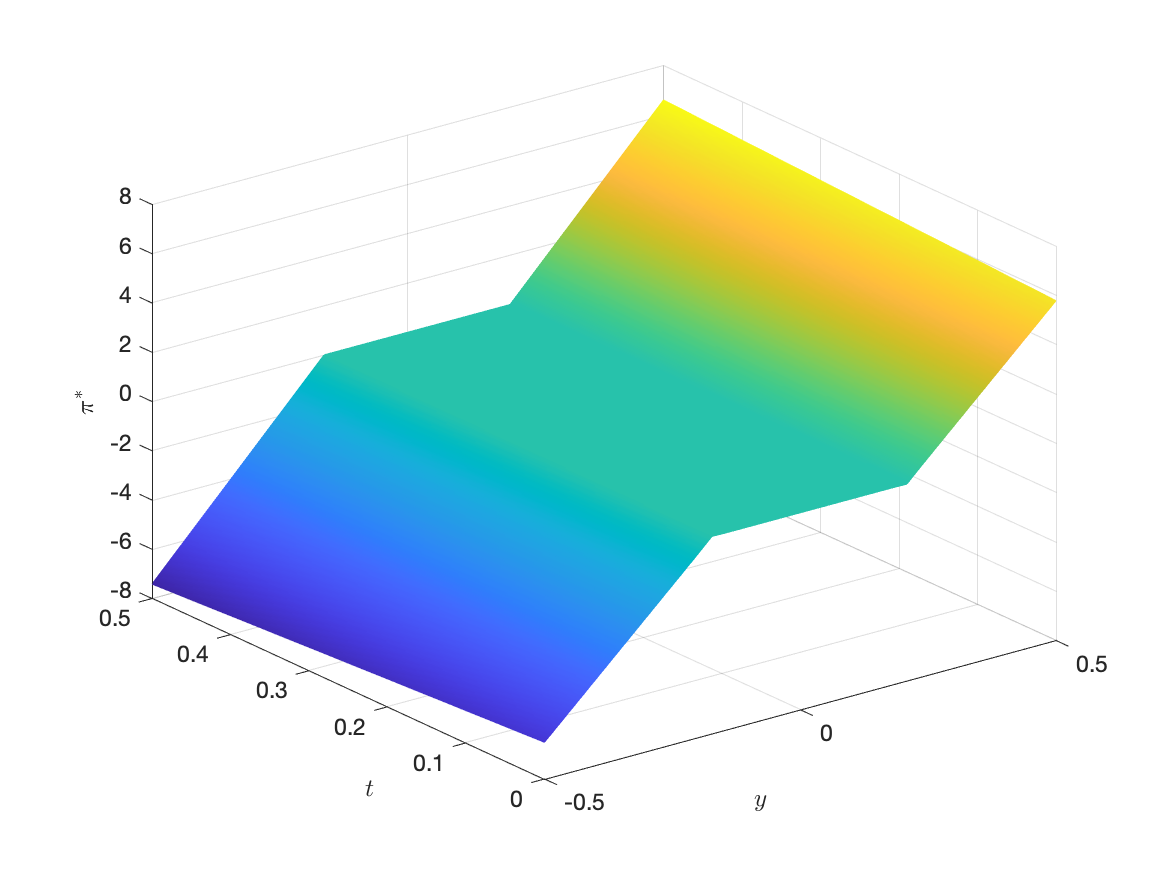}}
	\caption{The robust optimal feedback function $\hat\pi(t,y)$.}\label{f2}
\end{figure}



In Fig.~\ref{f1}, we illustrate the robust optimal feedback function \(\hat\pi(t,y)\), which reveals three distinct strategies across different ranges of \(y\), supporting the findings of Proposition \ref{prop}. Fig.~\ref{f2} further analyzes the behavior of \(\hat\pi(t,y)\). By combining Corollary \ref{coro:hatpi} with Fig.~\ref{f2}, we categorize the robust optimal feedback function into three regions: buying, selling, and small-trading. When \(\mu^{\text{min}}_{t,y}>r\), the smallest Sharpe ratio is positive, leading the investor to maintain a positive myopic demand for the risky asset. In cases where \(y\) is relatively large, myopic demand predominates over hedging demand, resulting in a buying position in the risky asset. Conversely, when \(y\) is near \(r+a\sqrt{\gamma(t)}\), the signs of hedging and myopic demands are opposite, leading to a small robust optimal feedback function. In contrast, when \(\mu^{\text{max}}_{t,y}< r\), the largest Sharpe ratio becomes negative, prompting the investor to adopt a negative myopic demand for the risky asset. Here, if \(y\) is relatively small, myopic demand prevails, leading to a selling position. Again, if \(y\) approaches \(r-a\sqrt{\gamma(t)}\), the signs of the hedging and myopic demands diverge, resulting in a small robust optimal feedback function.  When \(r \in \Lambda_{t,y}\), the investor only has hedging demand, and the robust optimal feedback function remains relatively small.

It is worth noting that Fig.~\ref{f1} shows that the middle region of $y$ narrows over time. Additionally, in Fig.~\ref{f2}, the robust optimal feedback function increases over time in the buying region, while it decreases in the selling region, indicating a more aggressive strategy as time progresses. This illustrates the effect of Bayesian learning, as the investor becomes less uncertain about the drift $\mu$ with the accumulation of information over time.


\subsection{Robust optimal investment strategy $\pi^*$}
\begin{figure}
	\centering
	{\includegraphics[width=0.8\linewidth]{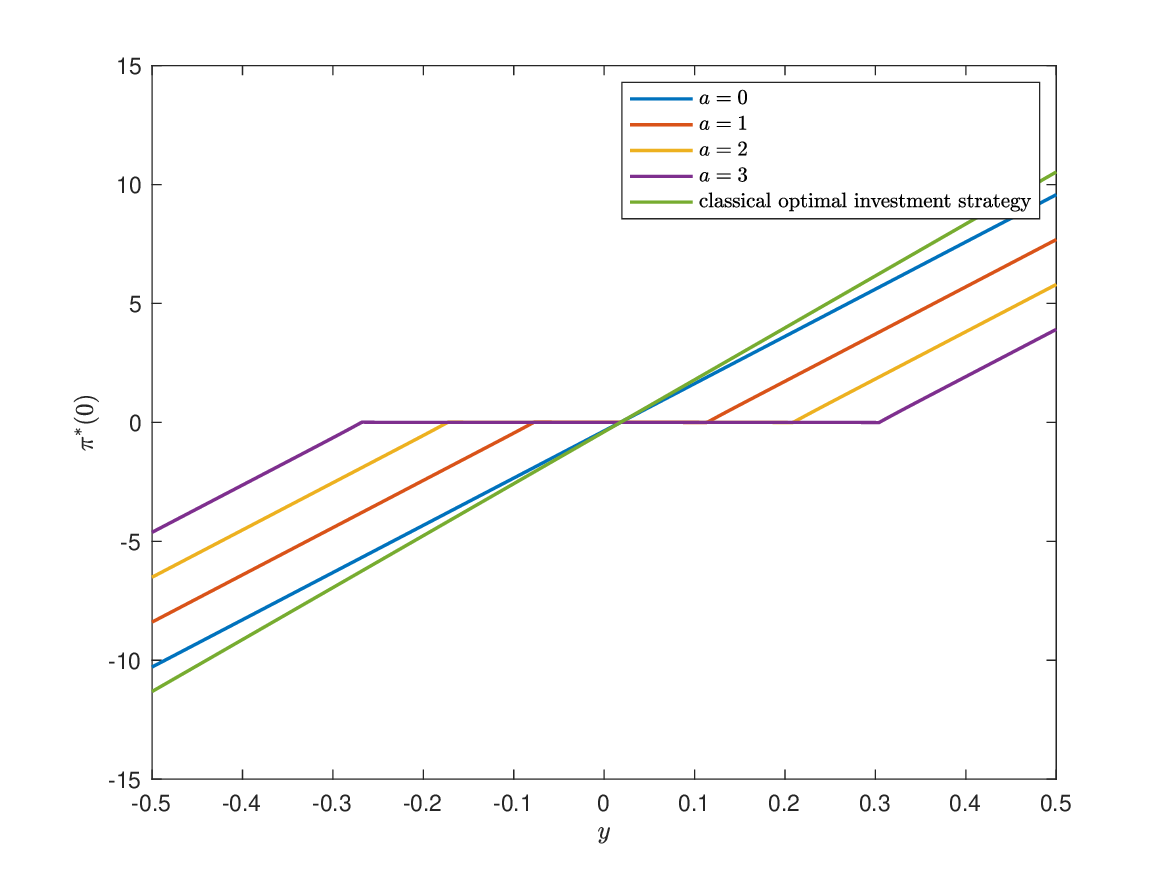}}
	\caption{The impacts of $a$ on the robust optimal investment strategy $\pi^{*}(0)$.} \label{f4}
\end{figure}

\begin{figure}
	\centering
	{\includegraphics[width=0.8\linewidth]{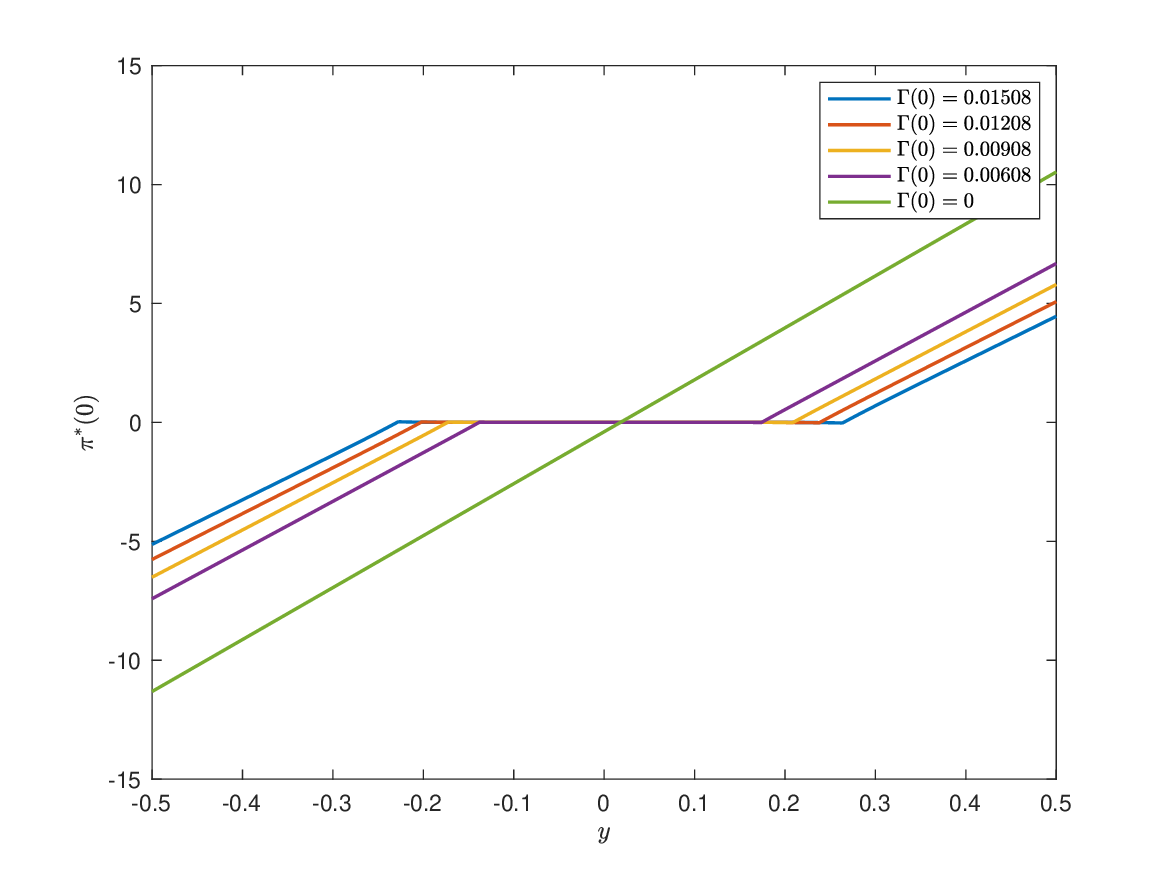}}
	\caption{The impacts of $\sigma_0^2$ on the robust optimal investment strategy $\pi^{*}(0)$.} \label{f5}
\end{figure}

\begin{figure}
	\centering
	{\includegraphics[width=0.8\linewidth]{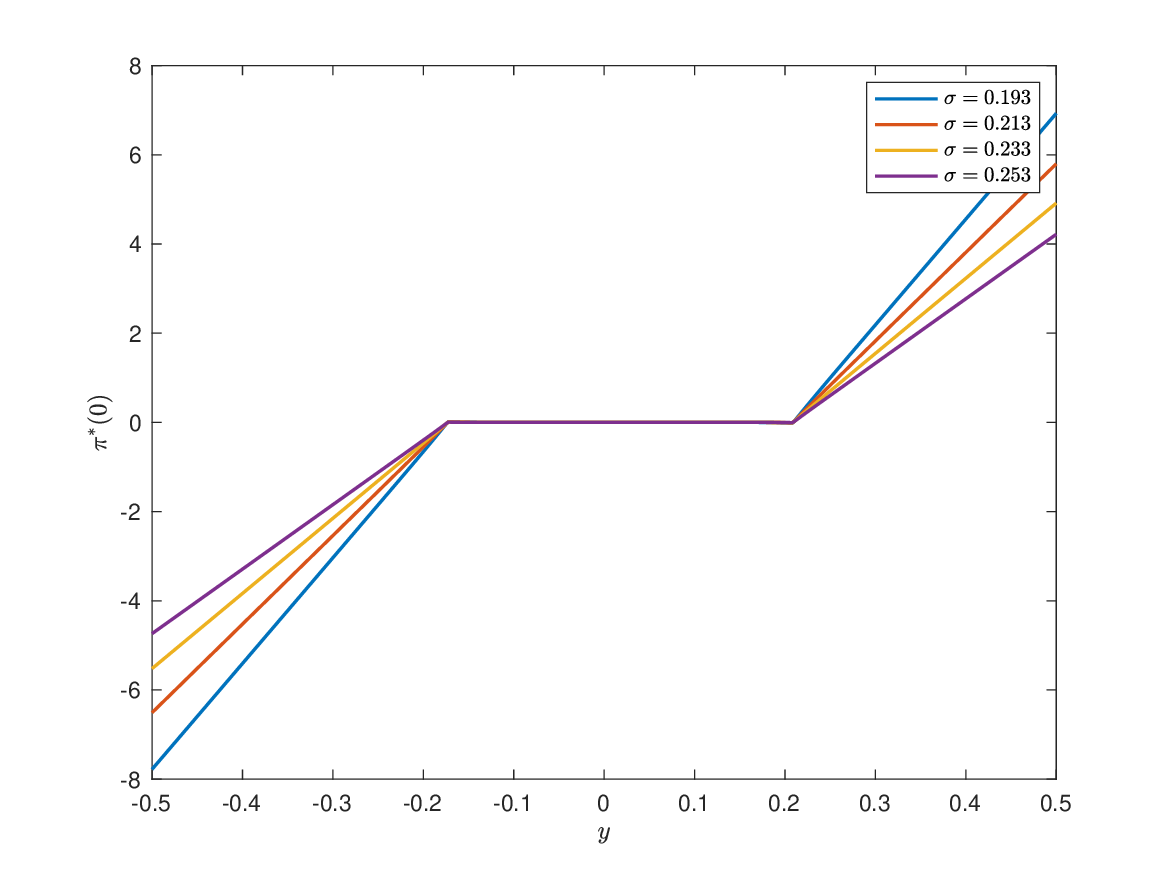}}
	\caption{The impacts of $\sigma$ on the robust optimal investment strategy $\pi^{*}(0)$.} \label{f6}
\end{figure}

In this subsection, we study the impacts of exogenous parameters on the robust optimal investment strategy $\pi^{*}(0)$ at time $0$. We also compare it with the classical optimal investment strategy $\bar{\pi}^{*}$ and the optimal investment strategy under partial information $\pi^{*}_0$.

In Fig.~\ref{f4}, we analyze the effect of the parameter \(a\) on the robust optimal investment strategy \(\pi^{*}(0)\). As \(a\) increases, the length of the confidence set for \(\mu\) expands, prompting the ambiguity-averse investor to adopt a more conservative stance, which results in a less aggressive robust optimal investment strategy \(\pi^{*}(0)\). Notably, the length of the small trading area decreases as \(a\) decreases, and it entirely disappears when \(a = 0\). Furthermore, the robust optimal investment strategy under partial information, \(\pi_0^{*}(0)\), incorporates a hedging demand and is less aggressive than the classical optimal investment strategy \(\bar{\pi}^{*}(0)\).

In Fig.~\ref{f5}, we examine the influence of \(\sigma_0^2\) on the robust optimal investment strategy \(\pi^{*}(0)\). A smaller variance \(\sigma_0^2\) reduces uncertainty around the drift \(\mu\), leading to a shorter confidence set for \(\mu\) and making risky investments more attractive to the ambiguity-averse investor. Consequently, the robust optimal investment strategy \(\pi^{*}(0)\) becomes more aggressive as \(\sigma_0^2\) decreases. Notably, when \(\sigma_0^2 = 0\), there is no uncertainty regarding \(\mu\), causing the robust optimal investment strategy to align with the classical optimal investment strategy.

In Fig.~\ref{f6}, we analyze the effect of \(\sigma\) on the robust optimal investment strategy \(\pi^{*}(0)\). As volatility \(\sigma\) increases, the length of the confidence set for \(\mu\) remains constant; however, the attractiveness of the risky investment return diminishes for the ambiguity-averse investor. Consequently, the robust optimal investment strategy \(\pi^{*}(0)\) becomes less aggressive as \(\sigma\) increases.

\section{Conclusions}
In this paper, we study an ambiguity-averse investor who is uncertain about the drift of a risky asset. The investor's belief about the unknown drift is updated through Bayesian learning. Based on the updated belief, we establish a state-dependent and time-dependent confidence set with a given confidence level. The investor seeks to maximize the expected utility of terminal wealth under the worst-case scenario for the unknown drift. We derive and solve the HJBI equation associated with this robust optimal investment problem, using its solution to determine a candidate robust optimal investment strategy. Specifically, the solution to the HJBI equation is represented by a PDE, which is a Cauchy problem for a one-dimensional linear second-order parabolic equation with unbounded coefficients, and we demonstrate the existence and uniqueness of this solution in our work. Additionally, we prove a verification theorem to confirm the optimality of the candidate robust optimal investment strategy and value function. Finally, we conduct numerical analyses to explore how ambiguity aversion and learning influence the optimal investment strategy. Our results show that ambiguity aversion concerning the risky asset's drift leads to a more conservative investment strategy, particularly when the conditional expectation of the drift is close to the risk-free interest rate. Furthermore, the robust investment strategy can be divided into three regions: buying, small trading, and selling. As the investor's uncertainty about the drift decreases over time, the optimal investment strategy becomes increasingly aggressive.

\paragraph{Acknowledgements.} The work is supported by the National Key R\&D Program of China
({2020YFA0712700}), 
the National Natural Science Foundation of China (12371477, 11901574,  12271290, {12071146}).
The authors thank the members of the group of Mathematical Finance and Actuarial Science at the Department of Mathematical Sciences, Tsinghua University for their helpful feedbacks and conversations. 



\bibliographystyle{abbrvnat}

\end{document}